\UseRawInputEncoding
\documentclass[12pt, reqno]{amsart}
\usepackage[paperheight=307mm, paperwidth=210mm, left=20mm, right=20mm]{geometry}
\usepackage{amssymb,latexsym,amsmath,amscd,amsfonts,amsthm,amsxtra}
\usepackage{latexsym}
\usepackage[mathscr]{eucal}
\usepackage{bm}
\usepackage{rotating}
\usepackage{pgfplots}
\usepackage{hyperref}
\usepackage{emptypage}
\usepackage{enumerate}
\usepackage{enumitem}

\newcommand{\C}{\mathbb{C}}

\newcommand{\Cn}{\mathbb{C}^n}
\newcommand{\D}{\mathbb{D}}
\newcommand{\T}{\mathbb{T}}

\newcommand{\G}{\widetilde{\mathbb{G}}_3}
\newcommand{\F}{\widetilde{\Gamma}_3}

\newcommand{\gn}{\mathbb{G}_n}
\newcommand{\Gn}{\widetilde{\mathbb{G}}_n}
\newcommand{\gamn}{\Gamma_n}
\newcommand{\Gamn}{\widetilde{\Gamma}_n}

\newcommand{\p}{\Phi_1}
\newcommand{\vp}{\psi}
\newcommand{\M}{\mathcal{M}}

\newcommand{\al}{\alpha}
\newcommand{\be}{\beta}
\newcommand{\lm}{\lambda}
\newcommand{\om}{\omega}

\newcommand{\q}{\quad}
\newcommand{\qq}{\qquad}
\newcommand{\n}{\lVert}
\newcommand{\imp}{\Rightarrow}
\newcommand{\Lra}{\Leftrightarrow}
\newcommand{\lf}{\left(}
\newcommand{\ls}{\left\{}
\newcommand{\lt}{\left[}
\newcommand{\rf}{\right)}
\newcommand{\rs}{\right\}}
\newcommand{\rt}{\right]}
\newcommand{\df}{\dfrac}

\newcommand{\la}{\langle}
\newcommand{\ra}{\rangle}
\newcommand{\nj}{{n \choose j}}

\def ~{\hspace{1mm}}

        \def \qed {\hfill \vrule height6pt width 6pt depth 0pt}
        \def\textmatrix#1&#2\\#3&#4\\{\bigl({#1 \atop #3}\ {#2 \atop #4}\bigr)}
        \def\dispmatrix#1&#2\\#3&#4\\{\left({#1 \atop #3}\ {#2 \atop #4}\right)}

\newtheorem{thm}{Theorem}[section]
\newtheorem{cor}[thm]{Corollary}
\newtheorem{lemma}[thm]{Lemma}

\numberwithin{equation}{section} \theoremstyle{definition}
\newtheorem{defn}[thm]{Definition}
\newtheorem{rem}[thm]{Remark}

\begin{document}

\title[A Schwarz lemma for ${\mathbb G_n}$]
{A Schwarz lemma for the symmetrized polydisc via estimates on another family of domains}

\author[Sourav Pal]{Sourav Pal}
\address[Sourav Pal]{Mathematics Department, Indian Institute of Technology Bombay, Powai, Mumbai - 400076, India.} 
\email{souravmaths@gmail.com , sourav@math.iitb.ac.in}

\author[Samriddho Roy]{Samriddho Roy}
\address[Samriddho Roy]{Mathematics Department, Indian Institute of Technology Bombay, Powai, Mumbai - 400076, India.} \email{samriddhoroy@gmail.com}



\thanks{The named author is supported by the Seed Grant of IIT Bombay, the CPDA of Govt. of India, the INSPIRE Faculty Award (Award No. DST/INSPIRE/04/2014/001462) of DST, India and the MATRICS Award of SERB, (Award No. SERB/F/10165/2019-2020) of DST, India. The second named author is supported by a Ph.D fellowship from the University Grand Commission of India.}

\begin{abstract}

We make some sharp estimates to obtain a Schwarz lemma for the \textit{symmetrized polydisc} $\mathbb G_n$, a family of domains naturally associated with the spectral interpolation, defined by
\[
\mathbb G_n :=\left\{ \left(\sum_{1\leq i\leq n} z_i,\sum_{1\leq
i<j\leq n}z_iz_j \dots, \prod_{i=1}^n z_i \right): \,|z_i|<1,
i=1,\dots,n \right \}.
\]
We first make a few estimates for the \textit{the extended symmetrized polydisc} $\widetilde{\mathbb G}_n$, a family of domains introduced in \cite{pal-roy 4} and defined in the following way:
\begin{align*}
\widetilde{\mathbb G}_n := \Bigg\{ (y_1,\dots,y_{n-1}, q)\in
\C^n :\; q \in \mathbb D, \;  y_j = \be_j + \bar \be_{n-j} q, \; \beta_j \in \mathbb C &\text{ and }\\
|\beta_j|+ |\beta_{n-j}| < {n \choose j} &\text{ for } j=1,\dots, n-1 \Bigg\}.
\end{align*}
We then show that these estimates are sharp and provide a Schwarz lemma for $\Gn$. It is easy to verify that $\mathbb G_n=\widetilde{\mathbb G}_n$ for $n=1,2$ and that ${\mathbb G}_n \subsetneq \widetilde{\mathbb G}_n$ for $n\geq 3$. As a consequence of the estimates for $\widetilde{\mathbb G_n}$, we have analogous estimates for $\mathbb G_n$. Since for a point $(s_1,\dots, s_{n-1},p)\in \mathbb G_n$, ${n \choose i}$ is the least upper bound for $|s_i|$, which is same for $|y_i|$ for any $(y_1,\dots ,y_{n-1},q) \in \widetilde{\mathbb G_n}$, $1\leq i \leq n-1$, the estimates become sharp for $\mathbb G_n$ too. We show that these conditions are necessary and sufficient for $\widetilde{\mathbb G_n}$ when $n=1,2, 3$. In particular for $n=2$, our results add a few new necessary and sufficient conditions to the existing Schwarz lemma for the symmetrized bidisc.
\end{abstract}

\maketitle


\section{Introduction}

\vspace{0.3cm}

\noindent This article is a sequel of our recent work \cite{pal-roy 4} and a next step to the pioneering works due to Agler-Young
\cite{AY-BLMS, AY-TAMS}, Bercovici \cite{hari, Hari-TAMS}, Bharali
\cite{bharali 07, bharali}, Costara \cite{costara1}, Dinen
\cite{Din}, Kosinski-Pflug-Zwonek \cite{KZ, zwonek1, zwonek5}, Nikolov \cite{NN, NTD, NW}, Nokrane-Ransford
\cite{NR1} and few others (see references there in). Throughout
the paper, $\mathbb{R}$ and $\mathbb{C}$ denote the set of real
numbers and the set of complex numbers respectively. For any
positive integer $n$, $\mathbb{C}^n$ is the cartesian product of
$n$ copies of $\mathbb{C}$. By $\mathbb{D}$, $\overline\D$ and
$\mathbb{T}$ we mean the open unit disc, the closed unit disc and
the unit circle with centre at the origin of $\mathbb{C}$
respectively. Following the standard convention, $\mathcal M_{m
\times n}(\mathbb C)$ (or $\C^{m \times n}$) represents the space
of all $m \times n$ complex matrices and $\mathcal M_n(\mathbb C)$
is used when $m=n$. For a matrix $A \in \mathcal M_n(\mathbb C)$,
$\n A \n$ is the operator norm of $A$. Also for a bounded linear
operator $T$, the spectrum and the spectral radius of $T$ are
denoted
by $\sigma(T)$ and $r(T)$ respectively.\\

The aim of this article is to produce a Schwarz type lemma for the symmetrized polydisc $\gn$, a family of domains given by
\[
\mathbb G_n :=\left\{ \left(\sum_{1\leq i\leq n} z_i,\sum_{1\leq
i<j\leq n}z_iz_j \dots, \prod_{i=1}^n z_i \right): \,|z_i|<1,
i=1,\dots,n \right \}.
\]
Clearly $\mathbb G_1=\D$ which is convex but for $n\geq 2$, $\gn$ is non-convex but polynomially convex (see \cite{ay-jfa, edi-zwo}). This family of domains is naturally associated with the spectral interpolation. Indeed, a matrix $A$ belongs to the spectral unit ball $\mathcal B_n^1 \subset \mathcal M_n(\C)$, i.e., $r(A)<1$  if and only if
$\pi_n(\lambda_1,\dots, \lambda_n) \in \mathbb G_n$ (see
\cite{costara1}). Here $\lambda_1, \dots , \lambda_n$ are
eigenvalues of $A$ and $\pi_n$ is the symmetrization map on
$\mathbb C^n$ defined by
\[
\pi_n(z_1,\dots, z_n) = \left(\sum_{1\leq i\leq n} z_i,\sum_{1\leq
i<j\leq n}z_iz_j,\dots, \prod_{i=1}^n z_i \right).
\]
For given distinct points $\xi_1,\dots, \xi_k$ in $\D$ and matrices $A_1,\dots , A_k \in \mathcal B_n^1$, the spectral interpolation seeks necessary and sufficient conditions under which there exists an analytic function $f:\D \rightarrow \mathcal M_n(\C)$ that interpolates the data, that is, $f(\xi_i)=A_i$ for $i=1,\dots ,k$. Apart from the derogatory matrices, the $n\times n$ spectral Nevanlinna-Pick problem is equivalent to a similar interpolation problem for $\mathbb G_n$ (see \cite{ay-ieot}, Theorem 2.1). Note that a bounded domain like $\mathbb G_n$, which has complex-dimension $n$, is much easier to deal with than an unbounded $n^2$-dimensional object like $\mathcal B_n^1$. The symmetrized polydisc has attracted considerable attention in past two decades for aspects of function theory \cite{ALY12, Bh-Sau, zwonek1, pal-roy 1, pal-roy 2, PZ, zwonek3}, complex geometry \cite{ALY14, AY04, bharali, costara, edi-zwo, NN, zwonek4, zwonek5} and operator theory \cite{ay-jfa, ay-jot, tirtha-sourav, tirtha-sourav1, Bisai-Pal1, sourav, sourav3, pal-shalit} (also see references there in).\\

It is evident that the symmetrized polydisc $\gn$ is the image of polydisc $\D^n$ under the symmetrization map $\pi_n$ which is holomorphic and proper. Agler and Young obtained a Schwarz type lemma in \cite{AY-BLMS} for the symmetrized bidisc $\mathbb G_2$ by providing a set of necessary and sufficient conditions each of which ensures the existence of an analytic map from $\D$ to $\mathbb G_2$ interpolating a two-point data. Similar estimates were made in \cite{NR1} by Nokrane and Ransford via an independent approach.\\

In this article, we make some sharp estimates to obtain a set of conditions each of which is necessary for the existence of an analytic function from $\D$ to $\gn$ interpolating a two-point data; $0, \lambda_0 \in \D$ map to $0, \underline s^0 \in \gn$ respectively. Novel work due to Nikolov, Pflug and Zwonek \cite{zwonek5} establishes the fact that the Lempert function of the symmetrized polydisc $\gn$ for $n\geq 3$ is not a distance, that is, $\gn$ does not satisfy the Lempert property (the coincidence of the Carath\'{e}odory pseudo distance and Lempert function) if $n\geq 3$. So, we do not expect to achieve a set of necessary and sufficient conditions in the Schwarz lemma for $\mathbb G_n$ when $n \geq 3$. But we have a different scenario for the symmetrized bidisc $\mathbb G_2$. In this case such conditions are both necessary and sufficient and the reason is that the Lempart's theorem  holds for $\mathbb G_2$ despite the fact that $\mathbb G_2$ is not convex. In fact, $\mathbb G_2$ is the first example of a non-convex domain which satisfies the Lempert property. For this reason the (necessary) conditions that we obtain for the Schwarz lemma for $\gn$ are also sufficient when $n=2$. Thus, we generalize the results of \cite{AY-BLMS, NR1} and also add to their account a few new necessary and sufficient conditions. \\

In \cite{costara1}, Costara found the following elegant description for the points in the symmetrized polydisc:
\[
\gn = \ls (s_1,\dots, s_{n-1}, p)\in \C^n : p\in \D,\; s_j=\be_j +
\bar\be_{n-j}p \: \text{ and } \: (\be_1, \dots, \be_{n-1}) \in
\mathbb G_{n-1} \rs .
\]
This characterization motivates us to introduce, \cite{pal-roy 4}, a new family of domains, namely the \textit{extended symmetrized polydisc} $\Gn$, which is defined in the following way:
\begin{align*}
 \widetilde{\mathbb G}_n : = \Bigg\{ (y_1,\dots,y_{n-1}, q)\in \C^n
:
|q|<1 , \:  y_j = \be_j + \bar \be_{n-j} q \: \text{ with }\;  |\beta_j| & +  |\beta_{n-j}| < {n \choose j},\\ & 1\leq j\leq n-1 \Bigg\}.
\end{align*}
It follows from Costara's description that for the points $(s_1,\dots, s_{n-1},p)\in \mathbb G_n$, $ {n \choose i}$ is the least upper bound for $|s_i|$. So, if $\lf \be_1,\dots, \be_{n-1} \rf \in \mathbb G_{n-1}$,
then $|\be_j| + |\be_{n-j}| < {n-1 \choose j}+{n-1 \choose n-j}= \nj $.  Therefore, it follows that $\gn \subseteq \Gn$. It is evident that $\widetilde{\mathbb G_1}={\mathbb G_1}=\D$ and we shall prove that $\widetilde{\mathbb G_2}={\mathbb G_2}$ and that $\mathbb G_n \subsetneq \widetilde{\mathbb G_n}$ for $n\geq 3$. Note that the least upper bound for $|y_i|$ for the points $(y_1,\dots, y_{n-1},q)\in \widetilde{\mathbb G_n}$ is also ${n \choose i}$. We make estimates to obtain a set of necessary conditions that provide a Schwarz lemma for $\Gn$. Then we prove that these estimates are sharp. We further show that these necessary conditions are sufficient also when $n=1,2,3$. Since $\gn \subseteq \Gn$ and the maximum modulus of every component of $\Gn$ and $\gn$ have the same least upper bound, these estimates become sharp for $\gn$ too. Consequently we obtain our desired Schwarz lemma for $\gn$. There are $n-1$ fractional linear transformations $\Phi_1,\dots ,\Phi_{n-1}$ which play central role in determining these conditions. We also show that an interpolating function in the Schwarz lemma for $\Gn$, when exists, is not unique. In fact, there are infinitely many such functions.\\

Also $\widetilde{\mathbb G}_3$ is linearly isomorphic to the tetrablock $\mathbb E$ by the map $(x_1,x_2,p) \rightarrow \left( \dfrac{x_1}{3}, \dfrac{x_2}{3}, p \right)$, where the tetrablock
\[
\mathbb E = \{(a,b,p)\in \C^3:\, |p|<1,\, a=c_1+\bar c_2 p\,,\, b=c_2+\bar c_1p \text{ with } |c_1|+|c_2|<1 \}
\]
is a polynomially convex domain which was originated in a special case of the $\mu$-synthesis problem (see \cite{awy}). In \cite{E-K-Z}, it was profoundly discovered by Edigarian, Kosinski and Zwonek that the Lempert theorem holds on the tetrablock. One can expect similar result to hold for $\widetilde {\mathbb G}_3$ (if one applies similar techniques as in \cite{E-K-Z}) as it is linearly isomorphic to $\mathbb E$. Nevertheless, we do not move into that direction here as our goal in this article is different. Since we have success of the Lempert theorem on $\Gn$ for $n=1,2$ and expect affirmative result naturally for $n=3$, this new family of domains $\Gn$ could possibly be a generalization of the symmetrized bidisc (as $\mathbb G_2=\widetilde {\mathbb G_2}$) in the sense of Lempert theorem, whereas the natural generalization $\gn$ of the symmetrized bidisc does not possess the Lempert property for $n>2$.\\

\noindent \textit{Plan of the paper:} We arrange our results in the following way. In Section 2, we recall from the literature a few characterizations for the points in $\Gn$ and $\overline{\Gn}$.  In Section 3, we derive a Schwarz type lemma for $\Gn$. In Section 4, we establish the nonuniqueness of interpolating function in the Schwarz lemma. In Section 5, we obtain a Schwarz lemma for $\gn$. \\

\noindent \textbf{Note.} The present article is the second part of authors' unpublished note \cite{pal-roy 3}. The first part of \cite{pal-roy 3} is published as \cite{pal-roy 4}. Since the tetrablock is linearly isomorphic to $\widetilde{\mathbb G_3}$, our results for $\widetilde{\mathbb G_n}$ are similar to that of \cite{awy} when $n=3$.\\

\noindent \textbf{Acknowledgement.} It is our pleasure to thank Professor W. Zwonek and Professor N. Nikolov profusely for being kind enough to read the first draft of this article and for making several helpful comments.

\section{Geometry of $\Gn$ and $\Gamn$}

\vspace{0.3cm}

\noindent In this section, we recall from the literature a few characterizations of the points in $\Gn$ and its closure $\Gamn$. It was proved in \cite{pal-roy 4} (see \cite{pal-roy 3} also) that the closure of $\Gn$ is the following subset of $\mathbb C^n$:
\begin{align*}
    \Gamn = \Bigg\{ (y_1,\dots,y_{n-1}, q)\in \C^n :\: q \in \D, \:  y_j = \be_j + \bar \be_{n-j} q, \: \beta_j \in \C \: & \text{ and }\\
    |\beta_j|+ |\beta_{n-j}| \leq {n \choose j},  & \q j=1,\dots, n-1 \Bigg\}.
\end{align*}
For studying the complex geometry of $\Gn$ and $\widetilde{\Gamma_n}$ we introduced in \cite{pal-roy 4}, $(n-1)$ fractional linear transformations $\Phi_1, \dots, \Phi_{n-1}$ in the following way.

\begin{defn}
For $z \in \C$, $y=(y_1,\dots,y_{n-1},q) \in \Cn $ and for any $j\in \left\{1,\dots,n-1\right\}$,
    let us define
    \begin{equation} \label{defn-p}
        \Phi_j(z,y) =
        \begin{cases}
            \dfrac{{n \choose j}qz-y_j}{y_{n-j}z-{n \choose j}} & \q \text{ if } y_{n-j}z\neq {n \choose j} \text{ and } y_j y_{n-j}\neq {n \choose j}^2 q \\
            \\
            \dfrac{y_j}{{n \choose j}} & \q \text{ if } y_j y_{n-j} = {n \choose j}^2 q \, ,
        \end{cases}
    \end{equation}
    \\ and
    \begin{equation} \label{defn-D}
        D_j(y) =\underset{z \in \D}{\sup}|\Phi_j (z,y)|= \n \Phi_j(.,y) \n _{H^\infty} \, ,
    \end{equation}\\
    where ${H^\infty}$ denotes the Banach space of bounded complex-valued analytic functions on $\D$ equipped with supremum norm.
\end{defn}

For $j \in \left\{1,\dots,n-1 \right\}$ and for a fixed $y=(y_1,\dots, y_{n-1},q)\in \Cn$, the function $\Phi_j(.,y)$ is a M\"{o}bius transformation. Note that the definition of $\Phi_j(.,y)$ depends only on three components of $y$, namely $y_j,\, y_{n-j}$ and $q$. Throughout the paper we denote by $\tilde y$ the point in $\C^n$ obtained by interchanging the coordinates $y_j$ and $y_{n-j}$ in $y$, that is,
\[
\tilde y = (\tilde y_1,\dots,\tilde y_{n-1},\tilde q), \quad \text{ where } \tilde q = q, \, \tilde y_i=y_i \text{ if } i \neq j,\, n-j \text{ and } \tilde y_i=y_{n-i} \text{ for } i=j.
\]


\noindent It was shown, in \cite{pal-roy 4}, that
\begin{equation} \label{formula-D}
    D_j(y)=\begin{cases}
        \dfrac{{n \choose j} \left|y_j - \bar y_{n-j} q \right| + \left|y_j y_{n-j} - {n \choose j}^2 q \right|}{{n \choose j}^2 - |y_{n-j}|^2 } & \q \text{ if } |y_{n-j}| < {n \choose j} \text{ and } y_j y_{n-j}\neq {n \choose j}^2 q  \\ \\
        \dfrac{|y_j|}{{n \choose j}} & \q \text{ if } y_j y_{n-j} = {n \choose j}^2 q \\ \\
    \end{cases}
\end{equation}
The following theorem provides a variety of characterizations for the points in $\Gn$ and $\Gamn$ and for proof to this result a reader is referred to \cite{pal-roy 4} or the unpublished note \cite{pal-roy 3} by the authors.

\begin{thm}[\cite{pal-roy 4}, Theorems 2.5 \& 2.7 ]  \label{char G 3}
    For a point $y =(y_1,\dots, y_{n-1},q) \in \Cn$, the following are equivalent:
    \begin{enumerate}
        \item[$(1)$] $y \in \Gn \quad (\text{respectively, } \in \Gamn )$;\\
        \item[$(2)$] ${n \choose j} - y_j z - y_{n-j}w + { n \choose j} qzw \neq 0$, for all $z,w \in \overline\D \quad (\text{respectively, for all } z,w\in \D )$ and for all $j = 1, \dots, \left[\frac{n}{2}\right]$;\\
        \item[$(3)$] for $j = 1, \dots, \left[\frac{n}{2}\right] $, $\; \n \Phi_j(.,y)\n_{H^{\infty}} < 1 \q ($respectively, $\leq 1 \; )$ and if $y_jy_{n-j}= {n \choose j}^2 q $ then in addition, $|y_{n-j}|< {n \choose j} \q ($ respectively, $\leq {n \choose j}\; )$;\\
        \item[$(3)'$] for $j = 1, \dots, \left[\frac{n}{2}\right] $, $\; \n \Phi_{n-j}(.,y)\n_{H^{\infty}} < 1 \q ($respectively, $\leq 1 \; )$ and if $y_jy_{n-j}= {n \choose j}^2 q $ then in addition, $|y_{j}|< {n \choose j} \q ($ respectively, $\leq {n \choose j}\; )$;\\
        \item[$(4)$] for $j = 1, \dots, \left[\frac{n}{2}\right]$, $\;
        {n \choose j}\left|y_j - \bar y_{n-j} q\right| + \left|y_j y_{n-j} - {n \choose j}^2 q \right| < {n \choose j}^2 -|y_{n-j}|^2
\quad (\text{respectively, } \leq {n \choose j}^2 -|y_{n-j}|^2 \text{ and if } y_jy_{n-j}= \nj^2 q \text{ then, in addition } |y_{n-j}| \leq \nj)        $;\\
    \item[$(4)'$] for $j = 1, \dots, \left[\frac{n}{2}\right]$, $\;
        {n \choose j}\left|y_{n-j} - \bar y_{j} q\right| + \left|y_j y_{n-j} - {n \choose j}^2 q \right| < {n \choose j}^2 -|y_{j}|^2
\quad (\text{respectively, } \leq {n \choose j}^2 -|y_{j}|^2 \text{ and if } y_jy_{n-j}= \nj^2 q \text{ then, in addition } |y_j| \leq \nj)        $;\\
        \item[$(5)$] for $j = 1, \dots, \left[\frac{n}{2}\right]$, $\; |y_j|^2 - |y_{n-j}|^2 + {n \choose j}^2|q|^2 + 2{n \choose j}\left|y_{n-j} - \bar y_j q\right| < {n \choose j}^2 \quad \left( \text{respectively, } \leq {n \choose j}^2 \right) \\ \text{ and } \; |y_{n-j}|< {n \choose j} \quad \left( \text{respectively, } |y_{n-j}|\leq {n \choose j} \right)$;\\
         \item[$(5)'$] for $j = 1, \dots, \left[\frac{n}{2}\right]$, $\; |y_{n-j}|^2 - |y_{j}|^2 + {n \choose j}^2|q|^2 + 2{n \choose j}\left|y_{j} - \bar y_{n-j} q\right| < {n \choose j}^2 \quad \left( \text{respectively, } \leq {n \choose j}^2 \right) \\ \text{ and } \; |y_{j}|< {n \choose j} \quad \left( \text{respectively, } |y_{j}|\leq {n \choose j} \right)$;\\       
        \item[$(6)$] $|q|<1 \quad ( \text{respectively, } \leq 1) $ and $\; |y_j|^2 + |y_{n-j}|^2 - {n \choose j}^2|q|^2 + 2\left| y_jy_{n-j} - {n \choose j}^2 q \right| < {n \choose j}^2 \q \left( \text{respectively, } \leq {n \choose j}^2 \right)$ for all $j = 1, \dots, \left[\frac{n}{2}\right]$;\\
        \item[$(7)$] $\left| y_{n-j} - \bar y_j q \right| + \left| y_j -\bar y_{n-j} q\right| < {n \choose j} (1 - |q|^2) \quad \left( \text{respectively, } \leq {n \choose j}(1-|q|^2) \right)$ for all $j = 1, \dots, \left[\frac{n}{2}\right]$;\\
        \item[$(8)$] there exist $2 \times 2$ matrices $B_1,\dots, B_{\left[\frac{n}{2}\right]}$ such that $\n B_j \n < 1 \; \; ( \text{respectively, } \leq 1)$, $y_j = {n \choose j}[B_j]_{11} $, $y_{n-j} = {n \choose j}[B_j]_{22}$ for all $j = 1, \dots, \left[\frac{n}{2}\right]$ and $\det B_1=  \dots = \det B_{\left[\frac{n}{2}\right]}= q$ ;\\
        \item[$(9)$] there exist $2 \times 2$ symmetric matrices $B_1,\dots, B_{\left[\frac{n}{2}\right]}$ such that $\n B_j \n < 1 \quad ( \text{respectively, } \leq 1)$, $y_j = {n \choose j}[B_j]_{11} $, $y_{n-j} = {n \choose j}[B_j]_{22}$ for all $j = 1, \dots, \left[\frac{n}{2}\right]$ and $\det B_1= \dots = \det B_{\left[\frac{n}{2}\right]}= q $.
    \end{enumerate}
\end{thm}

\section{A Schwarz lemma for $\Gn$}\label{section Schwarz lemma}

\vspace{0.3cm}

\noindent In this section, we shall make several sharp estimates and obtain a Schwarz type lemma for the extended symmetrized polydisc $\Gn$. We begin with an elementary lemma which will be used in the proof of the main result of this section.

\begin{lemma}\label{lemma-3.3}
    Let $\vp  : \D \longrightarrow \Gamn$ be an analytic function such that $\vp(\lm_0) \in \Gn$ for some $\lm_0 \in \D$, then $\vp(\D) \subset \Gn$.
\end{lemma}
\begin{proof}
    Write $\vp = (\vp_1,\dots,\vp_n)$. Suppose $\vp(\lm_0) \in \Gn$, for some $\lm_0 \in \D$. Then, by Theorem $\ref{char G 3}$, we have
    \[
    \Phi_j(\overline{\mathbb{D}},\vp(\lm_0)) \subset \D \q \text{and} \q |\vp_{n-j}(\lm_0)| < \nj , \q \q 1 \leq j \leq \left[\frac{n}{2}\right] .
    \]
Since the function $\vp : \D \longrightarrow \Gamn$ is analytic, the functions $\vp_{n-j}$ is also analytic on $\D$ and $|\vp_{n-j}(\lm)| \leq \nj$.
    Therefore, by Open mapping theorem, $|\vp_{n-j}(\lm)| < \nj$ for any $\lm \in \D$. Since $\vp(\lambda) \in \Gamn$, by Theorem \ref{char G 3}, we have $\n \Phi_j(., \vp(\lm))\n \leq 1$ for any $\lm \in \D$. Hence,
    \[
    \left| \Phi_j(z, \vp(\lm)) \right| \leq 1\q \text{for any } \lm \in \D, \, z \in \overline{\mathbb{D}}\;  \text{ and } 1 \leq j \leq \left[\frac{n}{2}\right] .
    \]
    For $z \in \overline{\mathbb{D}}$ and $j \in \ls 1, \dots,\left[\frac{n}{2}\right] \rs $, we define a function $g_z^j : \D \rightarrow \overline{\mathbb{D}}$ by $g_z^j(\lm)= \Phi_j(z, \vp(\lm))$. Clearly the function $g_z^j$ is analytic on $\D$ and $g_z^j(\lm_0)= \Phi_j(z, \vp(\lm_0)) \in \D$. Therefore, by Open mapping theorem, $g_z^j(\lambda) \in \D$ for all $\lambda \in \D$. 
    Consequently,  we have
    $ \n \Phi_j(.,\vp(\lm)) \n <1 $ for any $\lm \in \D$ and for any $j \in \ls 1, \dots,\left[\frac{n}{2}\right] \rs$.
    By Theorem \ref{char G 3}, we have $\vp(\lm) \in \Gn$ for any $\lm \in \D$. Hence $\vp(\D) \subset \Gn$.
\end{proof}

The following corollary is an immediate consequence of the preceding lemma.

\begin{cor}\label{Gn Gamn}
    Let $\lm_0 \in \D \; \backslash \; \{0\}$ and let $ y^0= (y_1^0,\dots,y_{n-1}^0,q^0) \in \Gn$. Then the following conditions are equivalent :
    \item[$(1)$]
    there exists an analytic function $\vp  :  \D \rightarrow  \Gamn $ such that  $ \; \vp(0) = (0,\dots,0) $ and $ \; \vp(\lm_0) = y^0 $;
    \item[$(2)$]
    there exists an analytic function $\vp  :  \D \rightarrow  \Gn $ such that  $\; \vp(0) = (0,\dots,0) $ and $\; \vp(\lm_0) = y^0 $.
\end{cor}

We state below a few preparatory results. For a strict $2\times 2$ matrix contraction $Z$ (i.e., $\|Z\|<1$), a matricial M\"{o}bius transformation $\M_Z$ is defined as
\[
\M_Z(X) = (1- ZZ^*)^{-\frac{1}{2}} (X-Z)(1- Z^*X)^{-1} (1-Z^*Z)^{\frac{1}{2}}\;, \quad X\in \C^{2\times 2} \text{ and } \n X \n <1.
\]
The map $\M_Z$ is an automorphism of the close unit ball of $\C^{2 \times 2}$ which maps $Z$ to the zero matrix and $\M_{Z}^{-1}=\M_{-Z}$.
\begin{lemma}[\cite{awy}, Lemma $3.1$]\label{lemma-3.1}
    Let $Z \in \C^{2\times 2}$ be such that $\n Z \n < 1$ and let $0 \leq \rho < 1$. Let
    \begin{equation}{\label{M-rho}}
        \mathcal{K}_Z(\rho) =\begin{bmatrix}
            [(1-\rho^2 Z^*Z)(1-Z^*Z)^{-1}]_{11} & [(1- \rho^2)(1 - ZZ^*)^{-1}Z]_{21} \\
            [(1 - \rho^2)Z^*(1 - ZZ^*)^{-1}]_{12} & [(ZZ^* - \rho^2)(1 - ZZ^*)^{-1}]_{22}
        \end{bmatrix}.
    \end{equation}
    \item[$(1)$]
    There exists $X\in \C^{2\times 2}$ such that $\n X \n \leq \rho$ and $[\M_{-Z} (X)]_{22} = 0$ if and only if $\det \mathcal{K}_Z(\rho) \leq 0$.
    \item[$(2)$]
    For any $2\times 2$ matrix X, $[\M_{-Z} (X)]_{22} = 0$ if and only if there exists $\al \in \C^2  \setminus \{0\} $ such that $$ X^*u(\al) = v(\al)$$
    where
    \begin{align}{\label{u-v}}
        & u(\al) = (1-ZZ^*)^{-\frac{1}{2}}(\al_1Ze_1 + \al_2e_2),\\
        \nonumber & v(\al) = -(1-Z^*Z)^{-\frac{1}{2}}(\al_1e_1 + \al_2Z^*e_2)
    \end{align}
    and $e_1 , e_2$ is the standard basis of $\C^2$.
\end{lemma}

\begin{lemma}[\cite{awy}, Lemma $3.2$]\label{lemma-3.2}
    Let $\lm_0 \in \D \setminus \{0\}$ let $Z \in \C^{2 \times 2}$ satisfy $\n Z \n < 1$ and let $\mathcal{K}_Z(\cdot)$ be given by equation $\eqref{M-rho}$,
    \item[$(1)$]
    There exists a function $G$ such that
    \begin{equation}{\label{G}}
        G \in S_{2 \times 2},\; \; [G(0)]_{22}=0 \text{ and } G(\lm_0)=Z
    \end{equation}
    if and only if $\det \mathcal{K}_Z(|\lm_0|)\leq 0$.
    \item[$(2)$]
    A function $G \in S_{2 \times 2}$ satisfies the conditions $\eqref{G}$ if and only if there
    exists $\al \in \C^2 \setminus \{0\}$ such that $\la\det \mathcal{K}_Z(|\lm_0|)\al, \al \ra \leq 0$ and a
    Schur function $Q$ such that $Q(0)^* \bar{\lm}_0u(\al)=v(\al)$ and $G = \M_{-Z}\circ (BQ)$,
    where $u(\al), v(\al)$ are given by equation $\eqref{u-v}$ and
    $B$ is the Blaschke factor
    \begin{equation}{\label{blaschke}}
        B(\lm) = \df{\lm_0 - \lm}{1 - \bar{\lm}_0\lm}.
    \end{equation}
\end{lemma}


Let $B_1,\dots, B_k$ be $2 \times 2$ contractive matrices such that $\det B_1= \det B_2 = \cdots = \det B_k$. We define two functions $\pi_{2k+ 1}$ and $\pi_{2k}$ in the following way:\\

$\pi_{2k+1} \lf B_1,\dots, B_k \rf 
= \Big( {n \choose 1} [B_1]_{11},  \dots , {n \choose k} [B_k]_{11},  {n \choose k} [B_k]_{22}, 
{n \choose k-1} [B_{k-1}]_{22}  \dots, {n \choose 1} [B_1]_{22} , \det B_1 \Big) $\\
and
\begin{align*}
    \pi_{2k} \lf B_1,\dots, B_k \rf
	=\Bigg( {n \choose 1} [B_1]_{11}, \dots , & {n \choose k-1} [B_{k-1}]_{11}, {n \choose k} \df{\lf [B_k]_{11} + [B_k]_{22} \rf}{2},\\ & {n \choose k-1} [B_{k-1}]_{22}, \dots, {n \choose 1} [B_1]_{22}, \det B_1  \Bigg).
\end{align*}
Then it is evident from part-(8) of Theorem \ref{char G 3} that
$\pi_{2k} \lf B_1,\dots, B_k \rf \in \widetilde{\mathbb G}_{2k}$ and $\pi_{2k+1} \lf B_1,\dots, B_k \rf \in \widetilde{\mathbb G}_{2k+1}$. For $ n\geq 3$, let $\mathcal J_n$ denote the following subset of $\Gn$:
\[
\mathcal J_n = 
\begin{cases}
\mathcal J_n^{odd} & \text{ if } n \text{ is odd}\\
\mathcal J_n^{even} & \text{ if } n \text{ is even},
\end{cases}
\]
where
\begin{align*}
	\mathcal J_n^{odd} = \Big\{ \lf y_1, \dots, y_{n-1}, y_n \rf \in \Gn : y_j = \dfrac{\nj}{n} y_1,\; y_{n-j} = \dfrac{\nj}{n}  y_{n-1}\, , \text{for } j=2,\dots, \left[\frac{n}{2} \right] \Big\}
\end{align*}
and 
\begin{align*}
	 \mathcal J_n^{even} = \Big\{ \lf y_1, \dots, y_{n-1}, y_n \rf \in \Gn & : y_{[\frac{n}{2} ]}= \dfrac{{n \choose [\frac{n}{2} ]}}{n} \df{y_1 + y_{n-1}}{2}\,,\;  y_j = \dfrac{\nj}{n} y_1 \,, \\ & y_{n-j} = \dfrac{\nj}{n}  y_{n-1},  \text{ for }j=2,\dots, \left[\frac{n}{2} \right]-1  \Big\}.
\end{align*}
It is merely mentioned that $\mathcal J_n=\Gn$ for $n=1,2,3$. For $ y= (y_1,\dots,y_{n-1},q)$ in $\Cn$, we consider the following two sets. For odd positive integer $n$ we have
\begin{align*}
    \mathcal F_1(y) = \Big\{ \tilde y = \lf \tilde y_1, \dots, \tilde y_{n-1}, \tilde q \rf \in &\C^n  :\;  \tilde q = \dfrac{q}{\lambda_0}, \text{ either } \tilde y_j = \dfrac{y_j}{\lambda_0} \text{ and } \tilde y_{n-j}= y_{n-j},\\
    & \text{ or } \tilde y_j = y_j \text{ and } \tilde y_{n-j} = \dfrac{y_{n-j}}{\lambda_0}, \text{ for } j=1,\dots,\lt \frac{n}{2} \rt \Big\} .
\end{align*}
For even positive integer $n$ we have
\begin{align*}
    \mathcal F_2(y) = \Bigg\{ \Bigg( \dfrac{n+1}{n} \tilde y_1, \dots, \dfrac{n+1}{\frac{n}{2}+1}\tilde y_{\frac{n}{2}},& \frac{n+1}{\frac{n}{2}+1} y_{\frac{n}{2}}, \dfrac{n+1}{\frac{n}{2}+2} \tilde y_{\frac{n}{2}+1}, \dots, \dfrac{n+1}{n} \tilde y_{n-1}, \tilde q \Bigg) \in \C^{n+1}:\\
    \tilde q = \dfrac{q}{\lambda_0}&, \; \tilde y_{\frac{n}{2}}= \dfrac{y_{\frac{n}{2}}}{\lambda_0} \text{ and for } j=1,\dots, \frac{n}{2}-1, \\
    \text{ either } \tilde y_j = \dfrac{y_j}{\lambda_0}& \; \text{ and } \tilde y_{n-j}= y_{n-j},\, \text{ or }\, \tilde y_j = y_j \text{ and } \tilde y_{n-j} = \dfrac{y_{n-j}}{\lambda_0} \Bigg\}.
\end{align*}
\noindent Note that for a point $\hat y =(\hat y_1,\dots,\hat y_n, \widehat q) \in\mathcal F_2(y)$ and for $j=1,\dots, \frac{n}{2}-1$, the $j$-th and $(n+1-j)$-th components of $\hat y$ are respectively the following:
\[\hat y_j = \dfrac{(n+1)}{(n+1-j)} \tilde y_j \,, \q \hat y_{n+1-j} = \dfrac{(n+1)}{(n+1-j)} \tilde y_{n-j}.\]
Also the $\frac{n}{2}$-th and $(\frac{n}{2}+1)$-th components of $\hat y$ are respectively the following:
\[
\hat y_{\frac{n}{2}}= \dfrac{n+1}{\frac{n}{2}+1}\tilde y_{\frac{n}{2}} \q \text{and} \q \hat y_{\frac{n}{2}+1}= \frac{n+1}{\frac{n}{2}+1} y_{\frac{n}{2}} \;.
\]
Below we define Schur class of functions which is essential in this context.
\begin{defn}
The \textit{Schur class} of type $m \times n$ is the set of analytic functions $F$ on $\mathbb{D}$ with values in the space $\mathbb{C}^{m\times n}$ such that $\lVert F(\lambda) \lVert \leq 1$ for all $\lambda \in \mathbb{D}$. Also we say that $F\in \mathcal{S}_{m \times n}$ if $\lVert F(\lambda) \lVert < 1$ for all $\lambda \in \mathbb{D}$.
\end{defn}

Being armed with the necessary results, we are now in a position to state and prove a Schwarz lemma for $\Gn$. This is the main result of this section.

\begin{thm}\label{Schwarz Gn}
    Let $\lm_0 \in \D \; \backslash \; \{0\}$ and let $ y^0= (y_1^0,\dots,y_{n-1}^0,q^0) \in \Gn$. Then in the set of following conditions, $(1)$ implies each of $(2)-(11)$. Moreover, $(3),(4),(6)-(11)$ are all equivalent.
    \item[$(1)$]
    There exists an analytic function $\vp  :  \D \rightarrow  \Gn $ such that  $\; \vp(0) = (0,\dots,0) $ and $\; \vp(\lm_0) = y^0 $.
    \item[$(2)$]
    \[
    \max_{1\leq j \leq n-1} \ls \n \Phi_j(.,y^0)\n_{H^{\infty}} \rs \leq |\lm_0| \, .
    \]

    \item[$(3)$]
    For each $j =1 , \dots,\left[\frac{n}{2}\right] $ the following hold
    \[
    \begin{cases}
    \n \Phi_j(.,y^0) \n_{H^{\infty}} \leq |\lm_0|  \q & \text{if } \; |y_{n-j}^0|\leq |y_j^0| \\
    \n \Phi_{n-j}(.,y^0) \n_{H^{\infty}} \leq |\lm_0| \q & \text{if } \; |y_j^0|\leq |y_{n-j}^0| \, .
    \end{cases}\]

    \item[$(4)$]
    If $n$ is odd and $\tilde y = \lf \tilde y_1, \dots, \tilde y_{n-1}, \tilde q \rf \in \mathcal F_1(y^0)$, where
        \[
    \begin{cases}
    \tilde y_j = \dfrac{y_j^0}{\lambda_0} \text{ and } \tilde y_{n-j}= y_{n-j}^0 & \text{when } |y_{n-j}^0| \leq |y_j^0|\\
    \tilde y_j = y_j^0 \text{ and } \tilde y_{n-j} = \dfrac{y_{n-j}^0}{\lambda_0} & \text{otherwise}
    \end{cases}
    \]
    for each $j = 1 , \dots,\left[\frac{n}{2}\right]$, then $\tilde y \in \Gamn$. If $n$ is even and
    \[
    \hat y = \left( \frac{n+1}{n} \tilde y_1, \dots, \frac{n+1}{\frac{n}{2}+1}\tilde y_{\frac{n}{2}}, \frac{n+1}{\frac{n}{2}+1}  y_{\frac{n}{2}}^0, \frac{n+1}{\frac{n}{2}+2} \tilde y_{\frac{n}{2}+1}, \dots, \frac{n+1}{n} \tilde y_{n-1}, \tilde q \right) \in \mathcal F_2(y^0)\,,
    \]
     where for each $j = 1 , \dots,\frac{n}{2}-1$,
    \[
    \begin{cases}
    \tilde y_j = \dfrac{y_j^0}{\lambda_0} \text{ and } \tilde y_{n-j}= y_{n-j}^0 & \text{when } |y_{n-j}^0| \leq |y_j^0|\\
    \tilde y_j = y_j^0 \text{ and } \tilde y_{n-j} = \dfrac{y_{n-j}^0}{\lambda_0} & \text{otherwise,}
    \end{cases}
    \]
     then $\hat y \in \widetilde{\Gamma}_{n+1}$.\\

    \item[$(5)$]  There exist functions $F_1, F_2, \dots , F_{\left[\frac{n}{2}\right]}$ in the Schur class such that
    $F_j(0) =
    \begin{bmatrix}
    0 & * \\
    0 & 0
    \end{bmatrix} ,\:$
    and $\; F_j(\lm_0) = B_j$, for $j = 1, \dots, \left[\frac{n}{2}\right]$, where $\det B_1= \cdots = \det B_{[\frac{n}{2}]}= q^0$,
    $y_j^0 = \nj [B_j]_{11}$ and $y_{n-j}^0 = \nj [B_j]_{22}$

    \item[$(6)$]
    For each $j =1 , \dots,\left[\frac{n}{2}\right] $ the following hold
    \[
    \begin{cases}
    \dfrac{{n \choose j}\left|y_j^0 - \overline y_{n-j}^0 q^0\right| + \left|y_j^0 y_{n-j}^0 - {n \choose j}^2 q^0 \right|}{{n \choose j}^2 -|y_{n-j}^0|^2} \leq |\lambda_0|  \q & \text{if } \; |y_{n-j}^0|\leq |y_j^0| \\
    \dfrac{{n \choose j}\left|y_{n-j}^0 - \overline{y^0_j} q^0 \right| + \left|y_j^0 y_{n-j}^0 - {n \choose j}^2 q^0 \right|}{{n \choose j}^2 -|y_j^0|^2} \leq |\lambda_0| \q & \text{if } \; |y_j^0|\leq |y_{n-j}^0| \, .
    \end{cases}\]

    \item[$(7)$]
    For each $j =1 , \dots,\left[\frac{n}{2}\right] $ the following hold
    \[
    \begin{cases}
    \nj \lm_0 - y_j^0 z - y_{n-j}^0 \lm_0 w + \nj q^0zw \neq 0, \text{ for all } z,w \in \D  \q & \text{if } \; |y_{n-j}^0|\leq |y_j^0| \\
    \nj \lm_0 - y_{n-j}^0 z - y_j^0 \lm_0 w + \nj q^0zw \neq 0, \text{ for all } z,w \in \D \q & \text{if } \; |y_j^0|\leq |y_{n-j}^0| \, .
    \end{cases}\]

    \item[$(8)$]
    For each $j =1 , \dots,\left[\frac{n}{2}\right] $ the following hold
    \[
    \begin{cases}
    |y_j^0|^2 - |\lm_0|^2|y_{n-j}^0|^2 + {n \choose j}^2|q^0|^2 - {n \choose j}^2|\lm_0| ^2 + 2{n \choose j}\left||\lm_0|^2 y_{n-j}^0 - \bar y_j^0 q^0 \right| \leq 0  \;  & \text{if } \; |y_{n-j}^0|\leq |y_j^0| \\
    |y_{n-j}^0|^2 - |\lm_0|^2|y_j^0|^2 + {n \choose j}^2|q^0|^2 - {n \choose j}^2|\lm_0| ^2 + 2{n \choose j}\left||\lm_0|^2 y_j^0 - \bar y_{n-j}^0 q^0 \right| \leq 0 \; & \text{if } \; |y_j^0|\leq |y_{n-j}^0| \, .
    \end{cases}\]

    \item[$(9)$]
    $|q^0| \leq |\lm_0|$ and for each $j =1 , \dots,\left[\frac{n}{2}\right] $ the following hold
    \[
    \begin{cases}
    |y_j^0|^2 + |\lm_0|^2 |y_{n-j}^0|^2 - \nj^2 |q^0|^2 + 2 |\lm_0|\left| y_j^0 y_{n-j}^0 - \nj^2 q^0 \right| \leq \nj^2 |\lm_0|^2  \;  & \text{if } \; |y_{n-j}^0|\leq |y_j^0| \\
    |y_{n-j}^0|^2 + |\lm_0|^2 |y_j^0|^2 - \nj^2 |q^0|^2 + 2 |\lm_0|\left| y_j^0 y_{n-j}^0 - \nj^2 q^0 \right| \leq \nj^2 |\lm_0|^2  \; & \text{if } \; |y_j^0|\leq |y_{n-j}^0| \, .
    \end{cases}\]

    \item[$(10)$]
    For each $j =1 , \dots,\left[\frac{n}{2}\right] $ the following hold
    \[
    \begin{cases}
    \big||\lm_0|^2y_{n-j}^0 - \bar y_j^0 q^0 \big| +|\lm_0| \left| y_j^0 - \bar y_{n-j}^0 q^0 \right| + \nj |q^0|^2 \leq \nj |\lm_0|^2  \;  & \text{if } \; |y_{n-j}^0|\leq |y_j^0| \\
    \big||\lm_0|^2y_j^0 - \bar y_{n-j}^0 q^0 \big| +|\lm_0| \left| y_{n-j}^0 - \bar y_j^0 q^0 \right| + \nj |q^0|^2 \leq \nj |\lm_0|^2 \; & \text{if } \; |y_j^0|\leq |y_{n-j}^0| \, .
    \end{cases}\]

    \item[$(11)$]
    If $n$ is odd number, then $|q| \leq |\lm_0|$ and there exist $(\be_1,\dots,\be_{n-1}) \in \C^{n-1}$ such that for each $j =1 , \dots,\left[\frac{n}{2}\right] $, $|\be_j|+ |\be_{n-j}| \leq \nj$ and the following hold
    \[
    \begin{cases}
    y_j = \be_j \lm_0 + \bar\be_{n-j} q \q \text{ and } \q  y_{n-j} \lm_0 = \be_{n-j}\lm_0 + \bar\be_j q  \q & \text{if } \; |y_{n-j}^0|\leq |y_j^0| \\
    y_j \lm_0 = \be_j\lm_0 + \bar\be_{n-j} q \q\text{ and }\q y_{n-j} = \be_{n-j} \lm_0 + \bar\be_j q \q & \text{if } \; |y_j^0|< |y_{n-j}^0|
    \end{cases}\]
    and if $n$ is even number, then $|q| \leq |\lm_0|$ and there exist $(\be_1,\dots,\be_n) \in \C^n$ such that for each $j =1 , \dots,\frac{n}{2} $, $|\be_j|+ |\be_{n+1-j}| \leq \nj$  and the following hold
    \[
    \begin{cases}
    y_j = \be_j \lm_0 + \bar\be_{n+1-j} q \q \text{ and } \q y_{n-j} \lm_0 = \be_{n+1-j}\lm_0 + \bar\be_j q  \q & \text{if } \; |y_{n-j}^0|\leq |y_j^0| \\
    y_j \lm_0 = \be_j\lm_0 + \bar\be_{n+1-j} q \q\text{ and }\q y_{n-j} = \be_{n+1-j} \lm_0 + \bar\be_j q \q & \text{if } \; |y_j^0|\leq |y_{n-j}^0| \, .
    \end{cases}\]
Furthermore, if $ y^0 \in \mathcal J_n $ all the conditions $(1)-(11)$ are equivalent. So in particular the converse holds for $n=1,2,3$.
\end{thm}

\begin{proof}
    We show the following:
    \[
    \begin{array}[c]{ccccc}
    (1)&\imp&(2)&\imp&(3)\\
    &  & & &\Downarrow  \\
    &  & & & (5)
    \end{array}
    \q \text{ and } \q
    \begin{array}[c]{ccccc}
    (4)&\Lra&(3)&\Lra&(8)\\
    &  & \Updownarrow & & \\
    &  & (6)          & &
    \end{array}
    \q \text{ and } \q
    \begin{array}[c]{ccccc}
    &  & (9)          & & \\
    &  & \Updownarrow & & \\
    (7)&\Lra&(4)&\Lra&(10)\\
    &  & \Updownarrow & & \\
    &  & (11)          & &
    \end{array}
    \]
    \noindent $(1) \imp (2)$ :
    Suppose $\vp  :  \D \rightarrow  \Gn $ is an analytic map such that
    $\; \vp(0) = (0,\dots,0) $ and $\; \vp(\lm_0) = y^0 $. For $z \in \mathbb{D}$ and $ j \in \ls 1, \dots,n-1 \rs $, we define a map $g_z^j:\D \rightarrow \C$ by $g_z^j(\lm)= \Phi_j(z, \vp(\lm))$. Since $\vp(\lm) \in \Gn$, by Theorem \ref{char G 3}, $\n \Phi_j(.,\vp(\lm)) \n_{H^{\infty}} <1$. Thus, for any $z \in \mathbb{D}$ and for any $ j \in \ls 1, \dots,n-1 \rs $ the map  $g_z^j$ is an analytic self-map of $\D$. Also we have
    $ g_z^j(0) = \Phi_j\lf z,(0,\dots,0)\rf = 0. $
    Hence by classical Schwarz lemma (in one variable) we have that
    \[ | \Phi_j(z,y^0)| = |\Phi_j(z,\vp(\lm_0))| = |g_z^j(\lm_0)| \leq |\lm_0|.\]
    Since this is true for all $z \in \mathbb{D}$ and for any $ j \in \ls 1, \dots,n-1 \rs $, we have
    \[
    \n \Phi_j(.,y^0)\n_{H^{\infty}} \leq |\lm_0|, \q \q 1\leq j \leq n-1.
    \]
    Therefore,
    \[
    \max_{1\leq j \leq n-1} \ls \n \Phi_j(.,y^0)\n_{H^{\infty}} \rs \leq |\lm_0|
    \]
    and consequently $(2)$ holds.\\

    \noindent $(2) \imp (3)$: This is obvious.\\

    \noindent $(3) \imp (4)$:
    Suppose $(3)$ holds. Let
    \[ A = \ls j \in \ls 1, \dots,\lt \frac{n}{2} \rt \rs : |y_{n-j}^0| \leq |y_j^0| \rs \q \text{and} \q B =  j \in \ls 1, \dots,\lt \frac{n}{2} \rt \rs \setminus A  . \]
    If $j \in A$, then
    \begin{align}\label{phi A}
        \nonumber  & \left| \dfrac{\nj q^0 z - y_j^0}{y_{n-j}^0 z - \nj} \right| = \left| \Phi_j(z,y^0) \right| \leq |\lambda_0|  \q \text{for all } z\in \D \\
        & i.e., \left| \dfrac{\nj \dfrac{q^0}{\lambda_0} z - \dfrac{y_j^0}{\lambda_0}}{y_{n-j}^0 z - \nj} \right| \leq 1  \q \text{for all } z\in \D.
    \end{align}
    If $j \in B$, then $\left| \Phi_{n-j}(z,y^0) \right| \leq |\lambda_0|$ for all $z\in \D$ and similarly
    \begin{equation} \label{phi B}
        \left| \dfrac{\nj \dfrac{q^0}{\lambda_0} z - \dfrac{y_{n-j}^0}{\lambda_0}}{y_j^0 z - \nj} \right| \leq 1  \q \text{for all } z\in \D .
    \end{equation}
    \item[Case 1]: Let $n$ be odd positive integer. Let $\tilde y = \lf \tilde y_1, \dots, \tilde y_{n-1}, \tilde q \rf $,
where $\tilde q = \dfrac{q^0}{\lambda_0}$ and for $j = 1, \dots,\lt \frac{n}{2}\rt$,
    \[
    \tilde y_j =
    \begin{cases}
    \dfrac{y_j^0}{\lambda_0} & \text{if } j\in A\\
    \\
    y_j^0 & \text{if } j\in B
    \end{cases}
    \q \q \text{ and } \q \q
    \tilde y_{n-j}=
    \begin{cases}
    y_{n-j}^0 & \text{if } j\in A\\
    \\
    \dfrac{y_{n-j}^0}{\lambda_0} & \text{if } j\in B .
    \end{cases}
    \]
    Clearly $\tilde y \in \mathcal F_1(y^0)$. From equations $\eqref{phi A}$ and $\eqref{phi B}$, we have
    \begingroup
    \allowdisplaybreaks
    \begin{align*}
        &\n \Phi_j(.,\tilde y) \n_{H^{\infty}}= \left| \dfrac{\nj \dfrac{q^0}{\lambda_0} z - \dfrac{y_j^0}{\lambda_0}}{y_{n-j}^0 z - \nj} \right| \leq 1, \; \text{ if } j\in A \\
        \text{and} \qq &\n \Phi_{n-j}(.,\tilde y) \n_{H^{\infty}}= \left| \dfrac{\nj \dfrac{q^0}{\lambda_0} z - \dfrac{y_{n-j}^0}{\lambda_0}}{y_j^0 z - \nj} \right| \leq 1, \; \text{ if } j\in B.
    \end{align*}
    \endgroup
    If $j \in A$, then $|\tilde y_{n-j}| = |y_{n-j}^0| \leq \nj$ and if $j \in B$, then $|\tilde y_j| = |y_j^0| \leq \nj$. Therefore, by Theorem $\ref{char G 3}$, we conclude that $\tilde y \in \Gamn$.\\

    \item[Case 2]: Let $n$ be even positive integer. Consider the point
\[
 \hat y = \lf \dfrac{n+1}{n} \tilde y_1, \dots, \dfrac{n+1}{\frac{n}{2}+1}\tilde y_{\frac{n}{2}}, \dfrac{n+1}{\frac{n}{2}+1} y_{\frac{n}{2}}^0, \dfrac{n+1}{\frac{n}{2}+2} \tilde y_{\frac{n}{2}+1}, \dots, \dfrac{n+1}{n} \tilde y_{n-1}, \tilde q \rf\,,
\]
where $\tilde q = \dfrac{q^0}{\lambda_0}$, $\tilde y_{\frac{n}{2}}= \dfrac{y_{\frac{n}{2}}^0}{\lambda_0}$ and for $j = 1, \dots, \frac{n}{2} -1$,
    \[
    \tilde y_j =
    \begin{cases}
    \dfrac{y_j^0}{\lambda_0} & \text{if } j\in A \setminus \ls\frac{n}{2}\rs \\
    \\
    y_j^0 & \text{if } j\in B \setminus \ls\frac{n}{2}\rs
    \end{cases}
    \q \q \text{ and } \q \q
    \tilde y_{n-j}=
    \begin{cases}
    y_{n-j}^0 & \text{if } j\in A \setminus \ls\frac{n}{2}\rs\\
    \\
    \dfrac{y_{n-j}^0}{\lambda_0} & \text{if } j\in B \setminus \ls\frac{n}{2}\rs .
    \end{cases}
    \]
    Clearly $\hat y \in \mathcal F_2(y^0)$. Since $n$ is even, $\lt \frac{n}{2}\rt = \frac{n}{2}= \lt \frac{n+1}{2}\rt$ and so $n - \lt \frac{n}{2}\rt = \frac{n}{2}$. Also note that when $n$ is even, $y_{\frac{n}{2}}^0 = y_{(n - \frac{n}{2})}^0$ and hence $j = \frac{n}{2} \in A$. Then from $\eqref{phi A}$ we have
    \begingroup
    \allowdisplaybreaks
    \begin{align*}   
     & \left| \dfrac{{n+1 \choose \frac{n}{2}} \tilde q z - \hat y_{\frac{n}{2}}}{\hat y_{(\frac{n}{2}+1)} z - {n+1 \choose \frac{n}{2}}} \right| = 
    \left| \dfrac{{n \choose \frac{n}{2}} \dfrac{q^0}{\lambda_0} z - \dfrac{y_{\frac{n}{2}}^0}{\lambda_0}}{y_{\frac{n}{2}}^0 z - {n \choose \frac{n}{2}}} \right| \leq 1  \q \text{for all } z\in \D\\
    & \Lra  \left| \Phi_{\frac{n}{2}}(z,\hat y) \right| = \left| \dfrac{{n+1 \choose \frac{n}{2}} \tilde q z - \hat y_{\frac{n}{2}}}{\hat y_{(n+1-\frac{n}{2})} z - {n+1 \choose \frac{n}{2}}} \right| \leq 1  \q \text{for all } z\in \D.
    \end{align*}
    \endgroup
    Hence $\n \Phi_{\lt \frac{n+1}{2} \rt}(.,\hat y) \n_{H^{\infty}} =\n \Phi_{\frac{n}{2}}(.,\hat y) \n_{H^{\infty}} \leq 1$. Again for $j\in A \setminus \ls \frac{n}{2} \rs$, from the equations $\eqref{phi A}$, we have
    \begingroup
    \allowdisplaybreaks
    \begin{align*}
     \left| \Phi_j(z,\hat y) \right| = \left| \dfrac{{n+1 \choose j} \tilde q z - \hat y_j}{\hat y_{(n+1-j)} z - {n+1 \choose j}} \right| = \left| \dfrac{\nj \dfrac{q^0}{\lambda_0} z - \dfrac{y_j^0}{\lambda_0}}{y_{n-j}^0 z - \nj} \right| \leq 1 \q \text{for all } z\in \D.
    \end{align*}
    
    \endgroup
    Therefore $\n \Phi_j(.,\hat y) \n_{H^{\infty}} \leq 1$. Similarly for $j\in B $, from equation $\eqref{phi B}$, we have
    \begingroup
    \allowdisplaybreaks
    \begin{align*}
     \left| \Phi_{n+1-j}(.,\hat y) \right| = \left| \dfrac{{n+1 \choose j} \tilde q z - \hat y_{(n+1-j)}}{\hat y_j z - {n+1 \choose j}} \right| = \left| \dfrac{\nj \dfrac{q^0}{\lambda_0} z - \dfrac{y_{n-j}^0}{\lambda_0}}{y_j^0 z - \nj} \right| \leq 1  \q \text{for all } z\in \D.
    \end{align*}
    \endgroup
    Therefore $\n \Phi_{n+1-j}(.,\hat y) \n_{H^{\infty}} \leq 1$. Note that $\frac{n}{2} \notin B$ and $A \cup B = \ls 1, \dots,\lt \frac{n+1}{2} \rt \rs$. Hence $\n \Phi_j(.,\hat y) \n_{H^{\infty}} \leq 1$, if $j\in A$ and $\n \Phi_{n+1-j}(.,\hat y) \n_{H^{\infty}} \leq 1$, if $j\in B$. Therefore, by the equivalence of conditions $(1)$ and $(3)$ of Theorem $\ref{char G 3}$, we have $\hat y \in \widetilde{\Gamma}_{n+1}$.
    Hence $(4)$ holds.\\

    \noindent $(4) \imp (3)$:
    Suppose $(4)$ holds. Again we consider two cases.
    \item[Case 1]:
    Let $n$ be odd positive integer. Let $ \tilde y = \lf \tilde y_1, \dots, \tilde y_{n-1}, \tilde q \rf \in \mathcal F_1(y^0)$,
    where $\tilde y_j = \dfrac{y_j^0}{\lambda_0}$ for $j \in A$ and $\tilde y_{n-j}= \dfrac{y_{n-j}^0}{\lambda_0}$ for $j \in B$.
    Then $\tilde y \in \Gamn$.
    By Theorem $\ref{char G 3}$, we have $\n \Phi_k(.,\tilde y) \n_{H^{\infty}} \leq 1$ for all $k\in A$ and $\n \Phi_{n-k}(.,\tilde y) \n_{H^{\infty}} \leq 1$ for all $k\in B$.
    Then for $j \in A$, we have
    \begingroup
    \allowdisplaybreaks
    \begin{align*}
        \left| \Phi_j(z,\tilde y) \right|  =&\left| \dfrac{\nj \dfrac{q^0}{\lambda_0} z - \dfrac{y_j^0}{\lambda_0}}{y_{n-j}^0 z - \nj} \right| \leq 1  \q \text{for all } z\in \D\\
        \imp & \left| \Phi_j(z,y^0) \right| = \left| \dfrac{\nj q^0 z - y_j^0}{y_{n-j}^0 z - \nj} \right| \leq |\lambda_0|  \q \text{for all } z\in \D.
    \end{align*}
    \endgroup
    Therefore $\n \Phi_j(.,y^0) \n_{H^{\infty}} \leq |\lambda_0|$. Similarly, for $j \in B$, $ \left| \Phi_{n-j}(z,\tilde y) \right| \leq 1 $ for all $z\in \D$ implies that $ \n \Phi_{n-j}(.,y^0) \n_{H^{\infty}} \leq |\lambda_0|$. Therefore, for all $j \in \ls 1, \dots,\lt \frac{n}{2} \rt \rs$
    \[
    \begin{cases}
    \n \Phi_j(.,y^0) \n_{H^{\infty}} \leq |\lm_0|  \q & \text{if } \; |y_{n-j}^0|\leq |y_j^0| \\
    \n \Phi_{n-j}(.,y^0) \n_{H^{\infty}} \leq |\lm_0| \q & \text{if } \; |y_j^0|\leq |y_{n-j}^0| \, .
    \end{cases}\]
    \item[Case 2]:
    Let $n$ be even and let
    \[
    \hat y = \lf \dfrac{n+1}{n} \tilde y_1, \dots, \dfrac{n+1}{\frac{n}{2}+1}\tilde y_{\frac{n}{2}}, \dfrac{n+1}{\frac{n}{2}+1} y_{\frac{n}{2}}^0, \dfrac{n+1}{\frac{n}{2}+2} \tilde y_{\frac{n}{2}+1}, \dots, \dfrac{n+1}{n} \tilde y_{n-1}, \tilde q \rf \in \mathcal F_2(y^0)\;,
    \]
    where $\tilde y_j = \dfrac{y_j^0}{\lambda_0}$ for $j \in A$ and $\tilde y_{n-j}= \dfrac{y_{n-j}^0}{\lambda_0}$ for $j \in B$.
    Then $\hat y \in \widetilde{\Gamma}_{n+1}$.
    Note that $\lt \frac{n+1}{2} \rt= \frac{n}{2}$ and $\frac{n}{2} \in A.$ Using Theorem $\ref{char G 3}$, we have $\n \Phi_k(.,\tilde y_k) \n_{H^{\infty}} \leq 1$ for all $k \in A$ and $\n \Phi_{n-k}(.,\hat y) \n_{H^{\infty}} \leq 1$ for all $k \in B$. Then
  \begin{align*}
        \n \Phi_{\lt \frac{n+1}{2} \rt}(.,\hat y) \n_{H^{\infty}} =& \n \Phi_{\frac{n}{2}}(.,\hat y) \n_{H^{\infty}} \leq 1 \\
        \Lra & \left| \Phi_{\frac{n}{2}}(z,y^0) \right| = \left| \dfrac{{n \choose \frac{n}{2}} q^0 z - y_{\frac{n}{2}}^0}{y_{\frac{n}{2}}^0 z - {n \choose \frac{n}{2}}} \right| \leq |\lambda_0|  \q \text{for all } z\in \D.
    \end{align*}
   That is, $ \n \Phi_{\frac{n}{2}}(.,y^0) \n_{H^{\infty}} \leq  |\lambda_0|$. For any $j \in A \setminus \{\frac{n}{2}\}$, we have
    \begin{align*}
        \left| \Phi_j(z,\hat y) \right|  = & \left| \dfrac{{n+1 \choose j} \tilde q z - \hat y_j}{\hat y_{n+1-j} z - {n+1 \choose j}} \right| \leq 1  \q \text{for all } z\in \D\\
        \imp & \left| \Phi_j(z,y^0) \right| =\left| \dfrac{\nj q^0 z - y_j^0}{y_{n-j}^0 z - \nj} \right| \leq |\lambda_0|  \q \text{for all } z\in \D.
    \end{align*}
    That is, $ \n \Phi_j(.,y^0) \n_{H^{\infty}} \leq |\lambda_0|$. Note that
    $B = \ls j \in \ls 1, \dots,\frac{n}{2} -1 \rs : \tilde y_{n-j} = \dfrac{y_{n-j}^0}{\lambda_0} \rs$. Similarly, for $j \in B$,
   \[
    \left| \Phi_{n+1-j}(z,\hat y) \right| = \left| \dfrac{{n+1 \choose j} \tilde q z - \hat y_{n+1-j}}{\hat y_j z - {n+1 \choose j}} \right| \leq 1  \q \text{for all } z\in \D,
    \]
    which implies $ \n \Phi_{n-j}(.,y^0) \n_{H^{\infty}} \leq |\lambda_0|$. Therefore for all $j \in \ls 1, \dots, \frac{n}{2} \rs$
    \[
    \begin{cases}
    \n \Phi_j(.,y^0) \n_{H^{\infty}} \leq |\lm_0|  \q & \text{if } \; j \in A \\
    \n \Phi_{n-j}(.,y^0) \n_{H^{\infty}} \leq |\lm_0| \q & \text{if } \; j \in B \, ,
    \end{cases}\]
    and hence $(3)$ holds.\\
    
    \noindent $(3) \imp (5)$: Suppose  $(3)$ holds. For each $j = 1 , \dots,\left[\frac{n}{2}\right] $, we shall construct a matrix-valued function $F_j$ in the Schur class such that
    $F_j(0) = \begin{bmatrix}
    0 & * \\
    0 & 0
    \end{bmatrix} $
    and $F_j(\lambda_0) = B_j$ where $\det B_j = q^0$, $[B_j]_{11} = \dfrac{y_j^0}{\nj}$ and $[B_j]_{22} = \dfrac{y_{n-j}^0}{\nj}$. Let $j \in \left\{1 , \dots,\left[\frac{n}{2}\right] \right\} $ be arbitrary and let $|y_{n-j}^0| \leq |y_j^0|$. Then by hypothesis, $\n \Phi_j(.,y^0) \n \leq |\lm_0|$.
    First consider the case of $y_j^0 y_{n-j}^0 = \nj^2 q^0$. Then
    \[
     \frac{|y_{n-j}^0|}{\nj} \leq \frac{|y_j|^0}{\nj} =  D_j(y) = \n \Phi_j (.,y^0) \n \leq |\lambda_0| <1.
    \]
    Then by classical Schwarz lemma, there exist analytic maps $f_j,g_j : \D \longrightarrow \D$ such that $f_j(0)=0$, $f_j(\lm_0)=\df{y_j^0}{\nj}$, $g_j(0)=0$ and $g_j(\lm_0)=\df{y_{n-j}^0}{\nj}$. Consider the function
    $$
    F_j(\lm)=
    \begin{bmatrix}
    f_j(\lm) & 0\\
    0 & g_j(\lm)
    \end{bmatrix} . $$
    Then $F_j$ is clearly a Schur function and has the following properties
    $$ F_j(0)=
    \begin{bmatrix}
    0 & 0 \\
    0 & 0
    \end{bmatrix} \q \text{and} \q
     F_j(\lm_0)=
    \begin{bmatrix}
    \df{y_j^0}{\nj} & 0 \\
    \\
    0 & \df{y_{n-j}^0}{\nj}
    \end{bmatrix} = B_j (\text{say}).$$
    Then $\det B_j = \df{y_j^0 y_{n-j}^0}{\nj^2}=q^0$. 
    Therefore the desired Schur function $F_j$ is constructed for the case of $y_j^0 y_{n-j}^0 = \nj q^0$.\\
    
    Now let us consider the case $y_j^0 y_{n-j}^0 \neq \nj q^0$. If we can construct a function $F_j \in S_{2 \times 2}$ such that
    \begin{equation}\label{formulae-F}
        F_j(0)=
        \begin{bmatrix}
            0 & *\\
            0 & 0
        \end{bmatrix}
        \text{  and }
        F_j(\lm_0)=
        \begin{bmatrix}
            \df{y_j^0}{\nj} & w_j \\
            \\
            \lm_0w_j & \df{y_{n-j}^0}{\nj}
        \end{bmatrix}=B_j,
    \end{equation}
    where
    $
    w_j^2 =\df{y_j^0 y_{n-j}^0 - \nj^2 q^0}{ \nj^2 \lm_0}
    $.
    Then $\det B_j=q^0 $. Consequently the issue will be resolved for this case.
    Consider the matrix
    \begin{equation}{\label{Z}}
        Z_j =
        \begin{bmatrix}
            \df{y_j^0}{\nj\lm_0} & w_j \\
            \\
            w_j & \df{y_{n-j}^0}{\nj}
        \end{bmatrix}.
    \end{equation}
    To construct a Schur function $F_j$ satisfying condition $\eqref{formulae-F}$, it is sufficient to find $G_j \in S_{2 \times 2}$ such that $ [G_j(0)]_{22}=0$ and $G_j(\lm_0)=Z_j$. Since the function
    $
    F_j(\lm)=G_j(\lm)\begin{bmatrix}
    \lm & 0\\
    0 & 1
    \end{bmatrix}
    $
    is in Schur class and satisfies
   \[
           F_j(\lm _0)
        =\begin{bmatrix}
            \df{y_j^0}{\nj} & w_j \\
            \\
            \lm_0 w_j & \df{y_{n-j}^0}{\nj}
        \end{bmatrix}
    \q \text{and} \q
    F(0)
    =\begin{bmatrix}
    0 & *\\
    0 & 0
    \end{bmatrix}.
    \]
    Thus our aim is to find a $G_j\in S_{2 \times 2}$ such that $ [G_j(0)]_{22}=0$ and $G_j(\lm_0)=Z_j$. Lemma \ref{lemma-3.2}
    gives an equivalent condition for the existence of such $G_j$, provided that $\n Z_j \n < 1$. We have that
    \begin{align}\label{Dj to 5}
        &\nonumber \n \Phi_j(.,y^0) \n = \dfrac{{n \choose j}\left|y_j^0 - \bar y_{n-j}^0 q^0 \right| + \left|y_j^0 y_{n-j}^0 - {n \choose j}^2 q^0 \right|}{{n \choose j}^2 -|y_{n-j}^0|^2} \leq |\lambda_0| \\
        \imp & {n \choose j}\left|\dfrac{y_j^0}{\lm_0} - \bar y_{n-j}^0 \dfrac{q^0}{\lm_0}\right| + \left|\dfrac{y_j^0}{\lm_0} y_{n-j}^0 - {n \choose j}^2 \dfrac{q^0}{\lm_0} \right| \leq {n \choose j}^2 -|y_{n-j}^0|^2 .
    \end{align}
    
    Then, by the equivalence of conditions $(3)$ and $(5)$ of Theorem \ref{char G 3}, we have
    \begin{equation}\label{equ for Y-1}
        {n \choose j}^2 - \dfrac{|y_j^0|^2}{|\lm_0|^2} - | y_{n-j}^0 |^2 + {n \choose j}^2\dfrac{| q^0 |^2}{|\lm_0|^2} \geq 2\dfrac{\left| y_j^0 y_{n-j}^0 - {n \choose j}^2 q^0 \right| }{|\lm_0|}.
    \end{equation}
    Since $|y_{n-j}^0| \leq |y_j^0|$ and $\lm_0 \in \D$, we have
    \[
    \lf \dfrac{|y_j^0|^2 }{|\lm_0|^2} + |y_{n-j}^0|^2 \rf - \lf |y_j^0|^2 + \dfrac{|y_{n-j}^0|^2}{|\lm_0|^2} \rf = \lf \dfrac{1}{|\lm_0|^2} - 1 \rf \lf |y_j^0|^2 - |y_{n-j}^0|^2 \rf \geq 0.
    \]
    Hence
    \begingroup
    \allowdisplaybreaks
    \begin{align}\label{equ for Y-2}
        \nonumber {n \choose j}^2 - |y_j^0|^2 - \dfrac{| y_{n-j}^0 |^2}{|\lm_0|^2} + {n \choose j}^2\dfrac{| q^0 |^2}{|\lm_0|^2}
        \geq & {n \choose j}^2 - \dfrac{|y_j^0|^2}{|\lm_0|^2} - | y_{n-j}^0 |^2 + {n \choose j}^2\dfrac{| q^0 |^2}{|\lm_0|^2} \\
        \geq & 2\dfrac{\left| y_j^0 y_{n-j}^0 - {n \choose j}^2 q^0 \right| }{|\lm_0|}.
    \end{align}
    \endgroup
    For a $j \in \left\{1 , \dots,\left[\frac{n}{2}\right] \right\} $, let
    \begin{equation}{\label{Y-1}}
        X_j := \frac{|\lm_0|}{|y_j^0 y_{n-j}^0 - \nj^2 q^0|}\left( \nj^2 - |y_j^0|^2 -\frac{|y_{n-j}^0|^2}{|\lm_0|^2} + \frac{\nj^2 |q^0|^2}{|\lm_0|^2} \right).
    \end{equation}
    and
    \begin{equation}{\label{Y-2}}
        X_{n-j} := \frac{|\lm_0|}{|y_j^0 y_{n-j}^0 - \nj^2 q^0|}\left( \nj^2 - \frac{|y_j^0|^2}{|\lm_0|^2} -
        |y_{n-j}^0|^2  + \frac{\nj^2 |q^0|^2}{|\lm_0|^2} \right)
    \end{equation}
    Then, by inequality $\eqref{equ for Y-1}$ and $\eqref{equ for Y-2}$, we have $X_j \geq 2$ and $X_{n-j} \geq 2$.
    Note that if $\n \Phi_j(.,y) \n < |\lm_0|$, then the inequality in $\eqref{Dj to 5}$ and consequently inequalities $\eqref{equ for Y-1}$ and $\eqref{equ for Y-2}$ will be strict. As a result we have $X_j > 2$ and $X_{n-j} > 2$ whenever $\n \Phi_j(.,y) \n < |\lm_0|$. At this point we need two lemmas to complete the proof. We state the lemmas below and provide their proofs at the end of the paper.\\

    \begin{lemma}\label{lemma-3.5}
        Let $\lm_0 \in \D \setminus \{0\}$ and $y^0 \in \Gn$. Suppose $j \in \left\{1 , \dots,\left[\frac{n}{2}\right] \right\} $ and $y_j ^0 y_{n-j}^0 \neq \nj^2 q^0$. Also suppose $|y_{n-j}^0|\leq|y_j^0|$ and $\n \Phi_j(.,y) \n \leq |\lm_0|$. Let $Z_j$ be defined by equation \eqref{Z}, where $w_j^2 = \df{y_j ^0 y_{n-j}^0 - \nj^2 q^0}{\nj^2 \lm_0}$. Then $\n Z_j \n \leq 1$. Moreover, $\n Z_j \n = 1$ if and only if $\n \Phi_j(.,y) \n = |\lm_0|$.
    \end{lemma}

    \begin{lemma}\label{lemma-3.6}
        Let $\lm_0 \in \D \setminus \{0\}$ and $y^0 \in \Gn$. Suppose $j \in \left\{1 , \dots,\left[\frac{n}{2}\right] \right\} $ and $y_j ^0 y_{n-j}^0 \neq \nj^2 q^0$. Also suppose $|y_{n-j}^0|\leq|y_j^0|$ and $\n \Phi_j(.,y) \n \leq |\lm_0|$. Let $Z_j$ be defined by equation \eqref{Z}, where $w_j^2 = \df{y_j ^0 y_{n-j}^0 - \nj^2 q^0}{\nj^2 \lm_0}$ and corresponding $\mathcal{K}_{Z_j}(\cdot)$ is defined by $\eqref{M-rho}$. Then
        \begingroup
        \allowdisplaybreaks
        \begin{align*}
            & \mathcal{K}_{Z_j}(|\lm_0|)\det(1-Z_j^*Z_j) =\\
            & \nonumber
            \begin{bmatrix}
                1 - \dfrac{|y_j^0|^2}{\nj^2} - \dfrac{|y_{n-j}^0|^2}{\nj^2} + |q^0|^2  {\qq} {\qq}
                & (1-|\lm_0|^2)\left(w_j + \df{q^0}{\lm_0}\bar{w_j} \right) \\
                -\dfrac{|y_j^0 y_{n-j}^0 - \nj^2 q^0|}{\nj^2}\left(|\lm_0|+ \dfrac{1}{|\lm_0|}\right) & \\
                \\
                (1-|\lm_0|^2)\left(\bar{w_j} + \df{\bar{q}^0}{\bar{\lm}_0}w_j \right) &
                -|\lm_0|^2 + \dfrac{|y_j^0|^2}{\nj^2} + \dfrac{|y_{n-j}^0|^2}{\nj^2} - \dfrac{|q^0|^2}{|\lm_0|^2}\\
                &{\q}+ \dfrac{|y_j^0 y_{n-j}^0 - \nj^2 q^0|}{\nj^2}\left( |\lm_0| + \dfrac{1}{|\lm_0|}\right)
            \end{bmatrix}
        \end{align*}
        \endgroup
        and
        \begin{equation*}
            \det\left(\mathcal{K}_{Z_j}(|\lm_0|)\det(1-Z_j^*Z_j) \right) = -(k-k_j)(k-k_{n-j})
        \end{equation*}
        where
        \begingroup
        \allowdisplaybreaks
        \begin{align*}
            k &= 2\frac{|y_j^0 y_{n-j}^0 - \nj^2 q^0|}{\nj^2} \\
            k_j &= |\lm_0|\left( 1 - \dfrac{|y_j^0|^2}{\nj^2} -\dfrac{|y_{n-j}^0|^2}{\nj^2|\lm_0|^2} + \dfrac{|q^0|^2}{|\lm_0|^2} \right) = \dfrac{|y_j^0 y_{n-j}^0 - \nj^2 q^0|}{\nj^2}X_j\\
            k_{n-j} &= |\lm_0|\left( 1 - \dfrac{|y_j^0|^2}{\nj^2|\lm_0|^2} -
            \dfrac{|y_{n-j}^0|^2}{\nj^2}  + \dfrac{|q^0|^2}{|\lm_0|^2} \right) = \dfrac{|y_j^0 y_{n-j}^0 - \nj^2 q^0|}{\nj^2}X_{n-j}.
        \end{align*}
        \endgroup
    \end{lemma}

\vspace{0.7cm}   

\noindent Let us get back to the proof of theorem. First we consider the case $\n \Phi_j(.,y) \n < |\lm_0|$.
    Then by Lemma \ref{lemma-3.5}, we have $\n Z_j \n < 1$ and so $\det (I - Z_j^*Z_j) >0$. Again since $X_j,X_{n-j} > 2$, by Lemma \ref{lemma-3.6}, we have
    $$\det\left(\mathcal{K}_{Z_j}(|\lm_0|)\det(I-Z_j^*Z_j) \right) = -\frac{|y_j^0 y_{n-j}^0 - \nj^2 q^0|^2}{\nj^4}(2 - X_j)(2 - X_{n-j}) < 0. $$
    Then $\det(\mathcal{K}_{Z_j}(|\lm_0|) < 0$. Thus by lemma \ref{lemma-3.2}, there exists
    $G_j \in S_{2 \times 2}$ such that $[G_j(0)]_{22}=0$ and $G_j(\lm_0)=Z_j$. Hence an $F_j$ with required properties is constructed
    and we have $(3) \imp (5)$ in the case  $|y_{n-j}^0|\leq|y_j^0|$. \\

    Now consider the case $\n \Phi_j(.,y) \n =|\lm_0|$. Take $\epsilon > 0$ so that $|\lm_\epsilon| < 1$,
    where $\lm_\epsilon = \lm_0(1+\epsilon)^2$. Then $\n \Phi_j(.,y) \n =|\lm_0|< |\lm_\epsilon| < 1$ and
    $$ \left(\df{w_j}{1+\epsilon}\right)^2 =\frac{y_j^0 y_{n-j}^0 - \nj^2 q^0}{\nj^2 \lm_\epsilon}. $$
    By same reason as described above, for each $\epsilon > 0$
    with $|\lm_\epsilon|< 1$, there exists $F_j^{\epsilon} \in S_{2\times 2}$ such that
    $$ F_j^{\epsilon}(0) = \begin{bmatrix}
    0 & *\\
    0 & 0
    \end{bmatrix} \q \text{and} \q
    F_j^{\epsilon}(\lm_\epsilon)=
    \begin{bmatrix}
    \dfrac{y_j^0}{\nj} & \dfrac{w_j}{1+\epsilon} \\
    \\
    (1+\epsilon)\lm_0w_j & \dfrac{y_{n-j}^0}{\nj}
    \end{bmatrix}.$$
    Since $ \n F_j^{\epsilon}(\lm) \n \leq 1$, the set of functions
    $$\{ F_j^{\epsilon} : \epsilon > 0 \; \text{ and } \; |\lm_\epsilon|< 1 \}$$
    is uniformly bounded on $\D$ and hence on each compact subsets of $\D$. So by Montel's theorem, there exists a subsequence of $\{F_j^{\epsilon} \}$ converging uniformly on each compact subset of $\D$, to an analytic function, say $F_j$, in the Schur class as $\epsilon \rightarrow 0$. Since $\lm_{\epsilon} \rightarrow \lm_0$ as $\epsilon \rightarrow 0$, we have $ F_j^{\epsilon}(\lm_\epsilon) \rightarrow F_j(\lm_0)$ as $\epsilon \rightarrow 0$. 
    Hence
    $$ F_j(0) = \begin{bmatrix}
    0 & *\\
    0 & 0
    \end{bmatrix} \q \text{and} \q
    F_j(\lm_0)=
    \begin{bmatrix}
    \dfrac{y_j^0}{\nj} & w_j \\
    \\
    \lm_0w_j & \dfrac{y_{n-j}^0}{\nj}
    \end{bmatrix}, $$
    where $w_j^2 = \df{y_j^0 y_{n-j}^0 - \nj^2 q^0}{\nj^2 \lm_0}$. This $F_j$ has the required properties.
    Thus we are done for the case $|y_{n-j}^0| \leq |y_j^0|$. The proof for the case $|y_j^0| \leq |y_{n-j}^0|$ is similar. Hence  $(3) \imp (5)$.\\

    \noindent $(3) \Lra (6)$: For all $j = 1 , \dots, n-1$, we have
    \[ \n \Phi_j(.,y^0)\n_{H^{\infty}}= \dfrac{{n \choose j}\left|y_j^0 - \bar y_{n-j}^0 q^0 \right| + \left|y_j^0 y_{n-j}^0 - {n \choose j}^2 q^0 \right|}{{n \choose j}^2 -|y_{n-j}^0|^2}.\]
    Therefore, the equivalence of $(3)$ and $(6)$ holds.\\

    Next we show that condition $(3)$ or condition $(4)$ is equivalent to the remaining conditions $(7)-(11)$. Consider the sets

    \noindent $(4) \Lra (7)$
    Suppose $n$ is odd positive integer and let $\tilde y = \lf \tilde y_1, \dots, \tilde y_{n-1}, \tilde q \rf \in \mathcal F_1(y^0)$, where
    \[
    \begin{cases}
    \tilde y_j = \dfrac{y_j^0}{\lambda_0} \text{ and } \tilde y_{n-j} = y_{n-j}^0 & \text{if } j\in A \\
    \tilde y_j = y_j^0 \text{ and } \tilde y_{n-j} = \dfrac{y_{n-j}^0}{\lambda_0} & \text{if } j \in B
    \end{cases}
    \]
    By hypothesis, $\tilde y \in \Gamn$. Note that $\tilde q = \dfrac{q^0}{\lambda_0}$.
    By condition $(2)$ of Theorem \ref{char G 3}, $\tilde y \in \Gamn$ if and only if for each $j \in \ls 1, \dots, \lt \frac{n}{2}\rt \rs$
    \[
    {n \choose j} - \tilde y_j z - \tilde y_{n-j}w + { n \choose j} \tilde qzw \neq 0 \; \text{ for all } z,w \in \D .
    \]
    Then, $\tilde y \in \Gamn$ if and only if for each $j \in \ls 1, \dots, \lt \frac{n}{2}\rt \rs$
    \[
    \begin{cases}
    {n \choose j}\lambda_0 - y_j^0 z - y_{n-j}^0 \lambda_0 w + { n \choose j} q^0 zw \neq 0,\; \text{ for all } z,w \in \D \q & \text{if } j\in A\\
    {n \choose j}\lambda_0 - y_j^0 \lambda_0 z -  y_{n-j}^0 w + { n \choose j} q^0 zw \neq 0,\; \text{ for all } z,w \in \D \q & \text{if } j \in B ,
    \end{cases}
    \]
    Thus we are done when $n$ is odd. Now suppose $n$ is even positive integer and
    \[
    \hat y =  \lf \dfrac{n+1}{n} \tilde y_1, \dots, \dfrac{n+1}{\frac{n}{2}+1}\tilde y_{\frac{n}{2}}, \dfrac{n+1}{\frac{n}{2}+1} y_{\frac{n}{2}}^0, \dfrac{n+1}{\frac{n}{2}+2} \tilde y_{\frac{n}{2}+1}, \dots, \dfrac{n+1}{n} \tilde y_{n-1}, \tilde q \rf \in \mathcal F_2(y^0) ,
    \]
    where, $\tilde q = \dfrac{q^0}{\lambda_0}$, $\tilde y_{\frac{n}{2}} = \dfrac{y_{\frac{n}{2}}^0}{\lambda_0}$ and for $j \in \ls 1, \dots, \frac{n}{2} -1 \rs$,
    \[
    \begin{cases}
    \tilde y_j = \dfrac{y_j^0}{\lambda_0} \text{ and } \tilde y_{n-j} = y_{n-j}^0 & \text{if } |y_{n-j}^0| \leq |y_j^0| \\
    \tilde y_j = y_j^0 \text{ and } \tilde y_{n-j} = \dfrac{y_{n-j}^0}{\lambda_0} & \text{otherwise}.
    \end{cases}
    \]
    By hypothesis, $\hat y \in \widetilde{\Gamma}_{n+1}$. By condition $(2)$ of Theorem \ref{char G 3}, $\hat y \in \widetilde{\Gamma}_{n+1}$ if and only if for each $j \in \ls 1, \dots, \lt \frac{n+1}{2}\rt \rs$
    \[
    {n+1 \choose j} - \hat y_j z - \hat y_{n+1-j}w + {n+1 \choose j} \widehat qzw \neq 0 \; \text{ for all } z,w \in \D .
    \]
    Since $\hat y_j = \dfrac{(n+1)}{(n+1-j)} \tilde y_j$ and $\hat y_{n+1-j} = \dfrac{(n+1)}{(n+1-j)} \tilde y_{n-j}$ for $j \in \ls 1, \dots, \frac{n}{2} -1 \rs$, $\hat y \in \widetilde{\Gamma}_{n+1}$ if and only if
%
%
    \[  {n \choose \frac{n}{2}}\lambda_0 - y_{\frac{n}{2}}^0 z - y_{\frac{n}{2}}^0 \lambda_0 w + {n \choose \frac{n}{2}} q^0 zw \neq 0,\; \text{ for all } z,w \in \D \]
    and for each $j \in \ls 1, \dots, \frac{n}{2} -1 \rs$ and for all $z,w \in \D$,
    \[
    \begin{cases}
    {n \choose j}\lambda_0 - y_j^0 z - y_{n-j}^0 \lambda_0 w + { n \choose j} q^0 zw \neq 0, \q &\text{if } |y_{n-j}^0| \leq |y_j^0|\\
    {n \choose j}\lambda_0 - y_j^0 \lambda_0 z -  y_{n-j}^0 w + { n \choose j} q^0 zw \neq 0,\q &\text{otherwise}.
    \end{cases}
    \]
    Thus we have $(4) \Lra (7)$.\\

    \noindent $(3) \Lra (8)$:
    By Maximum principle, $\n \Phi_j(.,y^0) \n_{H^{\infty}} \leq |\lambda_0|$ holds if and only if
    $ \left| \Phi_j(z,y^0) \right|\leq |\lambda_0| \text{ for all } z\in \T $.
    Then for all $z \in \T$, we have that
    \begingroup
    \allowdisplaybreaks
    \begin{align*}
        & \dfrac{\left| {n \choose j}q^0 z - y_j^0 \right|}{\left|y_{n-j}^0 z-{n \choose j}\right|}\leq |\lambda_0|  \\
        \Lra& \left| \nj q^0 z - y_j^0 \right|^2 \leq |\lm_0|^2 \left| y_{n-j}^0 z - \nj \right|^2  \\
        \Lra& |y_j^0|^2 - |\lm_0|^2 |y_{n-j}^0|^2 + {\nj}^2|q^0|^2 - {\nj}^2 |\lm_0|^2 + 2 \nj \left| |\lm_0|^2 y_{n-j}^0 - \bar y_j^0 q^0 \right| \leq 0.
    \end{align*}
    \endgroup
 Similarly,
    \begingroup
    \allowdisplaybreaks
    \begin{align*}
        & \n \Phi_{n-j}(.,y^0) \n _{H^\infty} \leq |\lm_0| \q \textrm{for all } z\in\mathbb{T} \\
                \Lra & |y_{n-j}^0|^2 - |\lm_0|^2 |y_j^0|^2 + {\nj}^2 |q^0|^2 - {\nj}^2 |\lm_0|^2 + 2 \nj \left| |\lm_0|^2 y_j^0 - \bar y_{n-j}^0 q^0 \right| \leq 0.
    \end{align*}
    \endgroup
    Therefore, $(3) \Lra (8)$.\\

    \noindent $(4) \Lra (9)$
    Suppose $n$ is odd positive integer and $\tilde y = \lf \tilde y_1, \dots, \tilde y_{n-1}, \tilde q \rf \in \mathcal F_1(y^0)$, where
    \[
    \begin{cases}
    \tilde y_j = \dfrac{y_j^0}{\lambda_0} \text{ and } \tilde y_{n-j} = y_{n-j}^0 & \text{if } j\in A \\
    \tilde y_j = y_j^0 \text{ and } \tilde y_{n-j} = \dfrac{y_{n-j}^0}{\lambda_0} & \text{if } j \in B
    \end{cases}
    \]
    for each $j \in \ls 1, \dots, \left[\frac{n}{2}\right] \rs$. By hypothesis, $\tilde y \in \Gamn$. Note that $\tilde q = \dfrac{q^0}{\lambda_0}$.
    By condition $(6)$ of Theorem \ref{char G 3}, $\tilde y \in \Gamn$ if and only if $|\tilde q| \leq 1$ and for each $j \in \ls 1, \dots, \lt \frac{n}{2}\rt \rs$
    \[
    |\tilde y_j|^2 + |\tilde y_{n-j}|^2 - {n \choose j}^2|\tilde q|^2 + 2\left| \tilde y_j \tilde y_{n-j} - {n \choose j}^2 \tilde q \right| \leq {n \choose j}^2.
    \]
    Then, $\tilde y \in \Gamn$ if and only if $|q^0| \leq |\lambda_0|$ and for each $j \in \ls 1, \dots, \lt \frac{n}{2}\rt \rs$
    \[
    \begin{cases}
    |y_j^0|^2 + |\lm_0|^2 |y_{n-j}^0|^2 - \nj^2 |q^0|^2 + 2 |\lm_0|\left| y_j^0 y_{n-j}^0 - \nj^2 q^0 \right| \leq \nj^2 |\lm_0|^2  & \text{if } j\in A\\
    |y_{n-j}^0|^2 + |\lm_0|^2 |y_j^0|^2 - \nj^2 |q^0|^2 + 2 |\lm_0|\left| y_j^0 y_{n-j}^0 - \nj^2 q^0 \right| \leq \nj^2 |\lm_0|^2 & \text{if } j \in B.
    \end{cases}
    \]
    Thus, $(4) \Lra (9)$ when $n$ is odd.
    Next suppose $n$ is even and
    \[\hat y =  \lf \dfrac{n+1}{n} \tilde y_1, \dots, \dfrac{n+1}{\frac{n}{2}+1}\tilde y_{\frac{n}{2}}, \dfrac{n+1}{\frac{n}{2}+1} y_{\frac{n}{2}}^0, \dfrac{n+1}{\frac{n}{2}+2} \tilde y_{\frac{n}{2}+1}, \dots, \dfrac{n+1}{n} \tilde y_{n-1}, \tilde q \rf \in \mathcal F_2(y^0) , \]
    where, $\tilde q = \dfrac{q^0}{\lambda_0}$, $\tilde y_{\frac{n}{2}} = \dfrac{y_{\frac{n}{2}}^0}{\lambda_0}$ and for $j \in \ls 1, \dots, \frac{n}{2} -1 \rs$,
    \[
    \begin{cases}
    \tilde y_j = \dfrac{y_j^0}{\lambda_0} \text{ and } \tilde y_{n-j} = y_{n-j}^0 & \text{if } |y_{n-j}^0| \leq |y_j^0| \\
    \tilde y_j = y_j^0 \text{ and } \tilde y_{n-j} = \dfrac{y_{n-j}^0}{\lambda_0} & \text{otherwise}.
    \end{cases}
    \]
    By hypothesis, $\hat y \in \widetilde{\Gamma}_{n+1}$. By condition $(6)$ of Theorem \ref{char G 3}, $\hat y \in \widetilde{\Gamma}_{n+1}$ if and only if
    $|\widehat q| \leq 1$ and for each $j \in \ls 1, \dots, \lt \frac{n+1}{2}\rt \rs$
    \[
    |\hat y_j|^2 + |\hat y_{n-j}|^2 - {n+1 \choose j}^2|\widehat q|^2 + 2\left| \hat y_j \hat y_{n-j} - {n+1 \choose j}^2 \widehat q \right| \leq {n+1 \choose j}^2.
    \]
    Note that $\hat y_j = \dfrac{(n+1)}{(n+1-j)} \tilde y_j$ and $\hat y_{n+1-j} = \dfrac{(n+1)}{(n+1-j)} \tilde y_{n-j}$ for $j \in \ls 1, \dots, \frac{n}{2} -1 \rs$. Therefore, $\hat y \in \widetilde{\Gamma}_{n+1}$ if and only if $|q^0| \leq |\lambda_0|$,
    \[
    |y_{\frac{n}{2}}^0|^2 + |\lm_0|^2 |y_{\frac{n}{2}}^0|^2 - {n \choose \frac{n}{2}}^2 |q^0|^2 + 2 |\lm_0|\left| y_{\frac{n}{2}}^0 y_{\frac{n}{2}}^0 - {n \choose \frac{n}{2}}^2 q^0 \right| \leq \nj^2 |\lm_0|^2
    \]
    and for each $j \in \ls 1, \dots, \frac{n}{2}-1 \rs$
    \[
    \begin{cases}
    |y_j^0|^2 + |\lm_0|^2 |y_{n-j}^0|^2 - \nj^2 |q^0|^2 + 2 |\lm_0|\left| y_j^0 y_{n-j}^0 - \nj^2 q^0 \right| \leq \nj^2 |\lm_0|^2 & \text{if } |y_{n-j}^0| \leq |y_j^0| \\
    |y_{n-j}^0|^2 + |\lm_0|^2 |y_j^0|^2 - \nj^2 |q^0|^2 + 2 |\lm_0|\left| y_j^0 y_{n-j}^0 - \nj^2 q^0 \right| \leq \nj^2 |\lm_0|^2 & \text{otherwise }.
    \end{cases}
    \]
    Thus we have $(4) \Lra (9)$.\\

    \noindent $(4) \Lra (10)$
    Suppose $n$ is odd, and $\tilde y = \lf \tilde y_1, \dots, \tilde y_{n-1}, \tilde q \rf \in \mathcal F_1(y^0)$, where
    \[
    \begin{cases}
    \tilde y_j = \dfrac{y_j^0}{\lambda_0} \text{ and } \tilde y_{n-j} = y_{n-j}^0 & \text{if } j\in A \\
    \tilde y_j = y_j^0 \text{ and } \tilde y_{n-j} = \dfrac{y_{n-j}^0}{\lambda_0} & \text{if } j \in B
    \end{cases}
    \]
    for each $j \in \ls 1, \dots, \left[\frac{n}{2}\right] \rs$. By hypothesis, $\tilde y \in \Gamn$. Note that $\tilde q = \dfrac{q^0}{\lambda_0}$.
    By condition $(7)$ of Theorem \ref{char G 3}, $\tilde y \in \Gamn$ if and only if for each $j \in \ls 1, \dots, \lt \frac{n}{2}\rt \rs$ we have
    \[
    \left| \tilde y_{n-j} - \bar{\tilde y}_j \tilde q \right| + \left| \tilde y_j -\bar{\tilde y}_{n-j} \tilde q\right| \leq {n \choose j} (1 - |\tilde q|^2).
    \]
    Then, $\tilde y \in \Gamn$ if and only if for each $j \in \ls 1, \dots, \lt \frac{n}{2}\rt \rs$
    \[
    \begin{cases}
    \left||\lm_0|^2y_{n-j}^0 - \bar y_j^0 q^0 \right| +|\lm_0| \left| y_j^0 - \bar y_{n-j}^0 q \right| + \nj |q^0|^2 \leq \nj |\lm_0|^2 & \text{if } j\in A\\
    \left||\lm_0|^2y_j^0 - \bar y_{n-j}^0 q^0 \right| +|\lm_0| \left| y_{n-j}^0 - \bar y_j^0 q^0 \right| + \nj |q^0|^2 \leq \nj |\lm_0|^2 & \text{if } j \in B.
    \end{cases}
    \]
    Thus, $(4) \Lra (10)$ when $n$ is odd.
    Next suppose $n$ is even and
    \[\hat y =  \lf \dfrac{n+1}{n} \tilde y_1, \dots, \dfrac{n+1}{\frac{n}{2}+1}\tilde y_{\frac{n}{2}}, \dfrac{n+1}{\frac{n}{2}+1} y_{\frac{n}{2}}^0, \dfrac{n+1}{\frac{n}{2}+2} \tilde y_{\frac{n}{2}+1}, \dots, \dfrac{n+1}{n} \tilde y_{n-1}, \tilde q \rf \in \mathcal F_2(y^0) , \]
    where, $\tilde q = \dfrac{q^0}{\lambda_0}$, $\tilde y_{\frac{n}{2}} = \dfrac{y_{\frac{n}{2}}^0}{\lambda_0}$ and for $j \in \ls 1, \dots, \frac{n}{2} -1 \rs$,
    \[
    \begin{cases}
    \tilde y_j = \dfrac{y_j^0}{\lambda_0} \text{ and } \tilde y_{n-j} = y_{n-j}^0 & \text{if } |y_{n-j}^0| \leq |y_j^0| \\
    \tilde y_j = y_j^0 \text{ and } \tilde y_{n-j} = \dfrac{y_{n-j}^0}{\lambda_0} & \text{otherwise}.
    \end{cases}
    \]
    By hypothesis, $\hat y \in \widetilde{\Gamma}_{n+1}$. By condition $(7)$ of Theorem \ref{char G 3}, $\hat y \in \widetilde{\Gamma}_{n+1}$ if and only if for each $j \in \ls 1, \dots, \lt \frac{n+1}{2}\rt \rs$
    \[
    \left| \hat y_{n-j} - \bar{\hat y}_j \widehat q \right| + \left| \hat y_j -\bar{\hat y}_{n-j} \widehat q\right| \leq {n \choose j} (1 - |\widehat q|^2).
    \]
    Since $\hat y_j = \dfrac{(n+1)}{(n+1-j)} \tilde y_j$ and $\hat y_{n+1-j} = \dfrac{(n+1)}{(n+1-j)} \tilde y_{n-j}$ for $j \in \ls 1, \dots, \frac{n}{2} -1 \rs$, therefore, $\hat y \in \widetilde{\Gamma}_{n+1}$ if and only if
    \[
    \left||\lm_0|^2y_{\frac{n}{2}}^0 - \bar y_{\frac{n}{2}}^0 q^0 \right| +|\lm_0| \left| y_{\frac{n}{2}}^0 - \bar y_{\frac{n}{2}}^0 q^0 \right| + \nj |q^0|^2 \leq \nj |\lm_0|^2
    \]
    and for each $j \in \ls 1, \dots, \frac{n}{2}-1 \rs$
    \[
    \begin{cases}
    \left||\lm_0|^2y_{n-j}^0 - \bar y_j^0 q^0 \right| +|\lm_0| \left| y_j^0 - \bar y_{n-j}^0 q^0 \right| + \nj |q^0|^2 \leq \nj |\lm_0|^2
    & \text{if } |y_{n-j}^0| \leq |y_j^0|\\
    \\
    \left||\lm_0|^2y_j^0 - \bar y_{n-j}^0 q^0 \right| +|\lm_0| \left| y_{n-j}^0 - \bar y_j^0 q^0 \right| + \nj |q^0|^2 \leq \nj |\lm_0|^2
    & \text{otherwise}.
    \end{cases}
    \]
    Thus we have $(4) \Lra (10)$.\\

    \noindent $(4) \Lra (11)$:
    Suppose $n$ is odd, and $\tilde y = \lf \tilde y_1, \dots, \tilde y_{n-1}, \tilde q \rf \in \mathcal F_1(y^0)$, where
    \[
    \begin{cases}
    \tilde y_j = \dfrac{y_j^0}{\lambda_0} \text{ and } \tilde y_{n-j} = y_{n-j}^0 & \text{if } j\in A \\
    \tilde y_j = y_j^0 \text{ and } \tilde y_{n-j} = \dfrac{y_{n-j}^0}{\lambda_0} & \text{if } j \in B
    \end{cases}
    \]
    for each $j \in \ls 1, \dots, \left[\frac{n}{2}\right] \rs$. By hypothesis, $\tilde y \in \Gamn$.
    Then by definition of $\Gamn$, $\tilde y \in \Gamn$ if and only if $\left|\dfrac{q^0}{\lambda_0}\right| \leq 1$ and there exist $\lf \be_1,\dots, \be_{n-1} \rf \in \C^{n-1}$ such that for each $j \in \ls 1, \dots, \lt \frac{n}{2}\rt \rs$,
    \[
    |\be_j|+ |\be_{n-j}| \leq \nj \q \text{and} \q \tilde y_j = \be_j + \bar \beta_{n-j} \tilde q , \q \tilde y_{n-j}= \be_{n-j} + \bar \beta_j \tilde q.
    \]
    That is, if and only if $|q^0| \leq |\lambda_0|$ and there exists $\lf \be_1,\dots, \be_{n-1} \rf \in \C^{n-1}$ such that for each $j \in \ls 1, \dots, \lt \frac{n}{2}\rt \rs$,we have $|\be_j|+ |\be_{n-j}| \leq \nj$ and
    \[
    \begin{cases}
    y_j^0 = \lambda_0 \be_j + \bar \be_{n-j} q^0 \q \text{and} \q y_{n-j}^0\lambda_0 = \be_{n-j} \lambda_0 + \bar \be_j q^0
    & \text{if } j\in A \\
    y_j^0 \lm_0 = \lambda_0 \be_j + \bar \be_{n-j} q^0 \q \text{and} \q y_{n-j}^0= \be_{n-j} \lambda_0 + \bar \be_j q^0
    & \text{if } j \in B.
    \end{cases}
    \]
    Thus, $(4) \Lra (11)$ if $n$ is odd.
    Next suppose $n$ is even number and
    \[\hat y =  \lf \dfrac{n+1}{n} \tilde y_1, \dots, \dfrac{n+1}{\frac{n}{2}+1}\tilde y_{\frac{n}{2}}, \dfrac{n+1}{\frac{n}{2}+1} y_{\frac{n}{2}}^0, \dfrac{n+1}{\frac{n}{2}+2} \tilde y_{\frac{n}{2}+1}, \dots, \dfrac{n+1}{n} \tilde y_{n-1}, \tilde q \rf \in \mathcal F_2(y^0) , \]
    where, $\tilde q = \dfrac{q^0}{\lambda_0}$, $\tilde y_{\frac{n}{2}} = \dfrac{y_{\frac{n}{2}}^0}{\lambda_0}$ and for $j \in \ls 1, \dots, \frac{n}{2} -1 \rs$,
    \[
    \begin{cases}
    \tilde y_j = \dfrac{y_j^0}{\lambda_0} \text{ and } \tilde y_{n-j} = y_{n-j}^0 & \text{if } |y_{n-j}^0| \leq |y_j^0| \\
    \tilde y_j = y_j^0 \text{ and } \tilde y_{n-j} = \dfrac{y_{n-j}^0}{\lambda_0} & \text{otherwise}.
    \end{cases}
    \]
    By hypothesis, $\hat y \in \widetilde{\Gamma}_{n+1}$.
    By definition, $\hat y \in \widetilde{\Gamma}_{n+1}$ if and only if $\left|\dfrac{q^0}{\lambda_0}\right| \leq 1$, and  there exist $\lf \gamma_1,\dots, \gamma_n \rf \in \C^n$ such that for each $j \in \ls 1, \dots, \lt \frac{n+1}{2} \rt \rs$,
    \[ |\gamma_j|+ |\gamma_{n+1-j}| \leq {n+1 \choose j} \q \text{and} \q  \hat y_j = \gamma_j + \bar \gamma_{n+1-j} \tilde q , \q \hat y_{n+1-j} = \gamma_{n+1-j} + \bar \gamma_j \tilde q .\]
    Since  $\hat y_j = \dfrac{(n+1)}{(n+1-j)} \tilde y_j$ and $\hat y_{n+1-j} = \dfrac{(n+1)}{(n+1-j)} \tilde y_{n-j}$,
    for $j \in \ls 1, \dots, \frac{n}{2} -1 \rs$, we have
    \[ \dfrac{(n+1)}{(n+1-j)} \tilde y_j = \gamma_j + \bar \gamma_{n+1-j} \tilde q \q \text{and} \q \dfrac{(n+1)}{(n+1-j)} \tilde y_{n-j} = \gamma_{n+1-j} + \bar \gamma_j \tilde q .\]
    Now set
    \[
    \be_j = \dfrac{(n+1-j)}{(n+1)} \gamma_j \q \text{and} \q \be_{n+1-j}= \dfrac{(n+1-j)}{(n+1)} \gamma_{n+1-j}, \q \text{for } j \in \ls 1, \dots, \frac{n}{2} -1 \rs.
    \]
    So, for $j \in \ls 1, \dots, \frac{n}{2} -1 \rs$ we have
    \[
    \tilde y_j = \be_j + \bar \beta_{n+1-j} \tilde q \q \text{and } \q \tilde y_{n-j}= \be_{n+1-j} + \bar \beta_j \tilde q,
    \]
    where $|\be_j|+ |\be_{n+1-j}| \leq \dfrac{(n+1-j)}{(n+1)} {n+1 \choose j} = \nj$. 
    Therefore, $(4)$ holds if and only if $\left|\dfrac{q^0}{\lambda_0}\right| \leq 1$, and there exists $\lf \be_1,\dots, \be_n \rf \in \C^n$ such that for $j \in \ls 1, \dots, \frac{n}{2} -1 \rs$,
    \[
    \begin{cases}
    y_j^0 = \lambda_0 \be_j + \bar \be_{n+1-j} q^0 \q \text{and} \q y_{n-j}^0\lambda_0 = \be_{n+1-j} \lambda_0 + \bar \be_j q^0 & \text{if } |y_{n-j}^0| \leq |y_j^0|\\
    y_j^0 \lm_0 = \lambda_0 \be_j + \bar \be_{n+1-j} q^0 \q \text{and} \q y_{n-j}^0= \be_{n+1-j} \lambda_0 + \bar \be_j q^0 & \text{otherwise}.
    \end{cases}
    \]
    Also we have,
    \[
    \dfrac{n+1}{\frac{n}{2}+1}\tilde y_{\frac{n}{2}} = \gamma_{\frac{n}{2}} + \bar \gamma_{\frac{n}{2}+1} \tilde q \q \text{and} \q \dfrac{n+1}{\frac{n}{2}+1} \tilde y_{\frac{n}{2}*} = \gamma_{\frac{n}{2} +1} + \bar \gamma_{\frac{n}{2}} \tilde q.
    \]
    Set, $\be_{\frac{n}{2}} = \dfrac{\frac{n}{2}+1}{n+1} \gamma_{\frac{n}{2}}$ and $\be_{\frac{n}{2}+1}= \dfrac{\frac{n}{2}+1}{n+1} \gamma_{\frac{n}{2}+1}$. Then
    \[
    \dfrac{y_{\frac{n}{2}}^0}{\lambda_0} = \be_{\frac{n}{2}} + \bar \beta_{\frac{n}{2}+1} \dfrac{q^0}{\lambda_0} \q \text{and } \q y_{\frac{n}{2}}^0= \be_{\frac{n}{2}+1} + \bar \beta_{\frac{n}{2}} \dfrac{q^0}{\lambda_0}.
    \]
    Therefore, $(4) \Lra (11)$. Thus we have proved the following:
    \[
    \begin{array}[c]{ccccc}
    (1)&\imp&(2)&\imp&(3)\\
    &  & & &\Downarrow  \\
    &  & & & (5)
    \end{array}
    \q \text{ and } \q
    \begin{array}[c]{ccccc}
    (4)&\Lra&(3)&\Lra&(8)\\
    &  & \Updownarrow & & \\
    &  & (6)          & &
    \end{array}
    \q \text{ and } \q
    \begin{array}[c]{ccccc}
    &  & (9)          & & \\
    &  & \Updownarrow & & \\
    (7)&\Lra&(4)&\Lra&(10)\\
    &  & \Updownarrow & & \\
    &  & (11)          & &
    \end{array}
    \]
It remains to prove that conditions $(1)-(11)$ are equivalent when $y^0 \in \mathcal J_n $. It suffices to prove $(5) \imp (1)$. Suppose $(5)$ holds, that is, there exist functions $F_1, F_2, \dots F_{\left[\frac{n}{2}\right]}$ in the Schur class such that
    $F_j(0) =
    \begin{bmatrix}
    0 & * \\
    0 & 0
    \end{bmatrix} \:$
    and $\; F_j(\lm_0) = B_j$, for $j = 1, \dots, \left[\frac{n}{2}\right]$, where $\det B_1= \cdots = \det B_{[\frac{n}{2}]}= q^0$,
    $y_j^0 = \nj [B_j]_{11}$ and $y_{n-j}^0 = \nj [B_j]_{22}$. Suppose $n$ is odd. Then $[B_j]_{11} = y_1/n$ and $[B_j]_{22} = y_{n-1}/n$ for $j = 1, \dots, \left[\frac{n}{2}\right]$. Therefore $F_1(\lm_0) = \dots = F_{\left[\frac{n}{2}\right]}(\lm_0)$. Set $\vp(\lm) = \pi_{2[\frac{n}{2}] +1} (\underbrace{ F_1(\lm), \dots, F_1(\lm)}_{[\frac{n}{2}]\text{-times}})$. Then $\vp$ is an analytic function from $\D$ to $\widetilde{\mathbb G}_{2[\frac{n}{2}] +1} = \Gn$ such that $ \vp(0) = (0,\dots,0)$ and $\vp(\lm_0) = y^0 $. If $n$ is even, then $[B_{[\frac{n}{2}]}]_{11} = \dfrac{y_1 + y_{n-1}}{2n} = [B_{[\frac{n}{2}]}]_{22}$, $[B_j]_{11} = y_1/n$ and $[B_j]_{22} = y_{n-1}/n$ for $j = 1, \dots, \left[\frac{n}{2}\right] -1$. Again by setting $\vp(\lm) = \pi_{2[\frac{n}{2}]} (\underbrace{ F_1(\lm), \dots, F_1(\lm) }_{[\frac{n}{2}]\text{-times}})$ we obtain an analytic function from $\D$ to $\widetilde{\mathbb G}_{2[\frac{n}{2}]} = \Gn$ such that $ \vp(0) = (0,\dots,0)$ and $\vp(\lm_0) = y^0 $. It is obvious that $\mathcal J_n = \Gn$ for $n=1,2,3$. Thus the conditions $(1)-(11)$ are all equivalent when $n=1,2,3$. The proof is now complete.
        
\end{proof}

\begin{rem}
For $n=2$, we add to the existing account (of \cite{AY-BLMS} and \cite{NR1}) several different necessary and sufficient conditions each of which ensures the existence of an analytic interpolant from $\D$ to $\mathbb G_2$.
\end{rem}

\section{Non-uniqueness}

\vspace{0.3cm}

\noindent In the previous section we have seen that each of the conditions
$(2)-(11)$ of Theorem \ref{Schwarz Gn} guarantees the existence of an analytic function $\vp:
\D \longrightarrow \Gn$ that maps the origin to the origin and
takes a point $\lm_0$ of $\D$ to a prescribed point $y^0 =
(y_1^0,\dots ,y_{n-1}^0,q^0) \in \Gn$. In this section, we shall show that the
interpolating function $\vp$, when exists, is not unique. We restrict our attention to the case $n=3$. We begin with a lemma which is a straightforward corollary of Theorem \ref{Schwarz Gn}.

\begin{lemma}\label{cor-Schwarz-lemma}
    Let $(y_1,y_2,q) \in \G$ and $|y_2| \leq |y_1|$. Then
    $$ \frac{|3y_2 - 3\bar{y}_1q|+|y_1y_2 - 9q|}{9 - |y_1|^2} \leq \frac{|3y_1 - 3\bar{y}_2q|+|y_1y_2 - 9q|}{9 - |y_2|^2}.$$
\end{lemma}

\begin{proof}
    If $(y_1,y_2,q)= (0,0,0)$ then clearly the inequality holds. For $(y_1,y_2,q)\neq (0,0,0)$ consider
    $$\lm_0  = \frac{|3y_1 - 3\bar{y}_2q|+|y_1y_2 - 9q|}{9 - |y_2|^2}.$$
    Since $(y_1,y_2,q) \in \G$, we have that
    $$|\lm_0|= \dfrac{|3y_1 - 3\bar{y}_2q|+|y_1y_2 - 9q|}{9 - |y_2|^2} = \n \p(.,(y_1,y_2,q)) \n < 1.$$
    Since $|y_2| \leq |y_1|$ condition $(3)$ of Theorem \ref{Schwarz Gn} is satisfied (for $n=3$). Hence condition $(2)$ of the same theorem, which is equivalent to condition $3$, holds and consequently we have the desired inequality.
    
\end{proof}

Let $y=(y_1,y_2,q)=\Bigg(\df{3}{2},\df{3}{4},\df{1}{2} \Bigg)$. Then, $y \in \G$ by part-(7) of Theorem \ref{char G 3}. Since $|y_2|<|y_1|$, we have by Lemma \ref{cor-Schwarz-lemma} that
\begin{align*}
    & \max\Bigg\{\frac{|3y_1 - 3\bar{y}_2q|+|y_1y_2 - 9q|}{9 - |y_2|^2},\frac{|3y_2 - 3\bar{y}_1q|+|y_1y_2 - 9q|}{9 - |y_1|^2}\Bigg\} \\
    & = \frac{|3y_1 - 3\bar{y}_2q|+|y_1y_2 - 9q|}{9 - |y_2|^2} = \frac{4}{5}.
\end{align*}
Let $\lm_0 = -\df{4}{5}$. We show that there are infinitely many analytic functions
$\vp : \D \longrightarrow  \F $ such that  $\vp(0) = (0,0,0) $ and $\vp(\lm_0) = y $.\\

Let $w^2 = \df{y_1y_2 - 9q}{9\lm_0} = \df{15}{32}$ and let
$$Z_y = \begin{bmatrix}
y_1/3\lm_0 & w \\
w & y_2/3
\end{bmatrix}
= \begin{bmatrix}
-\frac{5}{8} & w \\
w & \frac{1}{4}
\end{bmatrix}.$$
Since $y_1y_2 \neq 9q$ and $\n \p(.,y) \n = \dfrac{|3y_1 -
3\bar{y}_2q|+|y_1y_2 - 9q|}{9 - |y_2|^2} = |\lm_0|$, by Lemma
\ref{lemma-3.5} we have that $\n Z_y \n = 1$. The eigenvalues of
$Z_y$ are $-1$ and $5/8$, and
$
\begin{bmatrix}
\dfrac{8}{\sqrt{39}}w  \\
\dfrac{-3}{\sqrt{39}}
\end{bmatrix}
$
and
$
\begin{bmatrix}
\dfrac{4\sqrt{2}}{\sqrt{65}}w  \\
\dfrac{5\sqrt{2}}{\sqrt{65}}
\end{bmatrix}
$
are unit eigenvectors corresponding to the eigenvalues $-1$ and
$5/8$. Clearly these two eigenvectors are linearly independent.
Thus $Z_y$ is diagonalizable and can be written as
\[
Z_y  =
U_y\begin{bmatrix}
-1 & 0 \\
0 & 5/8
\end{bmatrix}U_y^*, \,
\text{ where the unitary matrix }
U_y =\begin{bmatrix}
\df{8}{\sqrt{39}}w & \df{4\sqrt{2}}{\sqrt{65}}w \\
\\
\df{-3}{\sqrt{39}} & \df{5\sqrt{2}}{\sqrt{65}}
\end{bmatrix}.
\]


Note that a function $G$ is a Schur Function if and only if $U^*GU$ is a Schur function. For each scalar function $g : \D \longrightarrow \overline{\mathbb{D}}$ satisfying $g(\lm_0)= \df{5}{8}$, we define a corresponding matrix valued function
$$H_g(\lm)
=\begin{bmatrix}
-1 & 0 \\
0 & g(\lm)
\end{bmatrix} ,\: \: \lm \in \D \; . $$
Then clearly $H_g$ is a Schur function. Now define a function by
$$G_g(\lm) = U_y H_g(\lm) U_y^* \q \text{for } \lm \in \D. $$
Since $U_y$ is unitary, we have $\n G_g(\lm) \n = \n H_g(\lm)\n $ for each $\lm \in \D$. Then $G_g$ is a Schur function and $G_g(\lm_0) =  Z_y$.


Note that,
\[
G_g(0) = U_y
\begin{bmatrix}
-1 & 0 \\
0 & g(0)
\end{bmatrix}
U_y^*
=\begin{bmatrix}
-\df{8}{\sqrt{39}}w & \df{4\sqrt{2}}{\sqrt{65}}w g(0)\\
\\
\df{3}{\sqrt{39}} & \df{5\sqrt{2}}{\sqrt{65}}g(0)
\end{bmatrix}
\begin{bmatrix}
\df{8}{\sqrt{39}}w & \df{-3}{\sqrt{39}} \\
\\
\df{4\sqrt{2}}{\sqrt{65}}w & \df{5\sqrt{2}}{\sqrt{65}}
\end{bmatrix},
\]
and thus \[\left[G_g(0)\right]_{22} = -\frac{9}{39} + \frac{50}{65}g(0)= \frac{1}{13}(10g(0)-3). \]
Hence for each scalar Schur function $g$ with
\begin{equation}\label{g}
    g(0) = 3/10 \q \text{and} \q g(-4/5) = 5/8,
\end{equation}
there is a Schur function $G_g$ satisfying
\begin{equation}\label{Gg}
    G_g(\lm_0)=Z_y \q \text{and} \q \big[G_g(0)\big]_{22} = 0.
\end{equation}
For each scalar Schur function $g$ satisfying \eqref{g}, we define another matrix valued function $F_g$
\[
F_g(\lm)
=G_g(\lm)
\begin{bmatrix}
\lm & 0\\
0 & 1
\end{bmatrix}
\q \text{for } \lm \in \D \;.
\]
Then $\n F_g(\lm) \n \leq 1$ for all $\lm \in \D$. Therefore, by Theorem \ref{char G 3}, $\lf [F_g(\lm)]_{11},[F_g(\lm)]_{22},\det F_g(\lm) \rf \in \F$. Note that
\[
F_g(\lm _0) =
\begin{bmatrix}
y_1/3 & w \\
\\
\lm_0w & y_2/3
\end{bmatrix} \;
\text{ and } \;
F_g(0)=
\begin{bmatrix}
0 & *\\
0 & 0
\end{bmatrix}.
\]


Finally for any scalar Schur function $g$ satisfying condition \eqref{g}, we define
\begin{align*}
    \vp_g  : \D & \longrightarrow  \F \\
    \lm & \longmapsto \lf [F_g(\lm)]_{11},[F_g(\lm)]_{22},\det F_g(\lm) \rf \; .
\end{align*}
Then $\vp_g$ is an analytic function such that $\vp_g(0) = (0,0,0)$ and
$$\vp_g(\lm_0) = \lf y_1, y_2, \dfrac{y_1 y_2}{9} - \lm_0 w^2 \rf = (y_1,y_2,q)=y.$$
Again
\[
F_g(\lm)
= U_y \begin{bmatrix}
-1 & 0 \\
0 & g(\lm)
\end{bmatrix}
U_y^*
\begin{bmatrix}
\lm & 0\\
0 & 1
\end{bmatrix} \q \text{for all } \lm \in \D.
\]
Since the matrix $U_y$ is unitary, the third component of $\vp_g$ is equal to $\det F_g(\lm)$ which is equal to $-\lm
g(\lm)$. Thus for a set of distinct functions $g$, the set of functions $\vp_g$ are distinct. The pseudo-hyperbolic distance between the points $\dfrac{3}{10}$ and $\dfrac{5}{8}$ is
\[d\left( \frac{3}{10} , \frac{5}{8} \right)= \frac{2}{5} < \frac{4}{5} = d\left(0, -\frac{4}{5} \right).\]
Hence there are infinitely many scalar Schur functions $g$ that satisfy \eqref{g}.

\section{A Schwarz lemma for $\gn$}

\vspace{0.3cm}

\noindent In this section, we present the desired Schwarz lemma for the symmetrized polydisc and this is the main result of this paper. The closure of $\gn$ is shown in \cite{costara1} to be the following set:
\[
\Gamma_n= \left\{ \left(\sum_{1\leq i\leq n} z_i,\sum_{1\leq
i<j\leq n}z_iz_j,\dots, \prod_{i=1}^n z_i \right): \,|z_i|\leq 1, \;
i=1,\dots,n \right \}.
\]
Needless to mention that $\Gamma_n$ is the image of the closed polydisc $\overline{\D^n}$ under the symmetrization map $\pi_n$.
We begin this section with the proof of the fact that $\gn=\widetilde{\mathbb G}_n$ for $n=1,2$ but $\gn \subsetneqq \Gn$ for $n\geq 3$. For that we need the following characterization theorem due to Costara, a part of which was mentioned in the Introduction.

\begin{thm}[\cite{costara1}, Theorems 3.6 \& 3.7]\label{gn}
	For a point $(s_1,\dots,s_{n-1},p) \in \C^n$, the following are equivalent:
	\begin{enumerate}
		\item The point $(s_1,\dots,s_{n-1},p) \in \gn$ (respectively $\in \Gamma_n$). 
		\item $|p|<1$ (respectively $\leq 1$) and there exists $\lf S_1,\dots, S_{n-1} \rf \in \mathbb G_{n-1}$ (respectively $\in \Gamma_{n-1}$) such that 
		\[ 
		s_j = S_j + \bar S_{n-j} p \q \text{for }\; j=1,\dots, n-1 .
		\]
	\end{enumerate}
\end{thm}

The following straight-forward result has appeared in \cite{pal-roy 4} in an informal way. For the sake of completeness we present it here in the form of a lemma.

\begin{lemma}
	$\mathbb G_2 = \widetilde{\mathbb{G}}_2$, but $\gn \subsetneqq \Gn$ for $n\geq 3$.
\end{lemma}

\begin{proof}
It is evident from definitions that $\mathbb G_n \subseteq \widetilde{\mathbb G_n}$
Let $(s,p) \in \mathbb G_2$. By Theorem 2.1 in \cite{AY04}, $|p|<1$ and $s = \be + \bar{\be} p$ for some $\beta \in \mathbb C$ with $|\be|<1$. So we have 
\[
\mathbb G_2 = \ls (\be + \bar{\be} p, p) : |\be|<1, |p|<1 \rs = \widetilde{\mathbb G_2}.
\]
Now we show that $\mathbb G_3 \subsetneqq \G$. It is evident from the definition that $(s,p)\in \mathbb G_2$, then $|s|<2$. Consider the point $\lf \dfrac{5}{2}, \dfrac{5}{4}, \dfrac{1}{2} \rf \in \C^3$. We can write $\lf \dfrac{5}{2}, \dfrac{5}{4}, \dfrac{1}{2} \rf = \lf \be_1 + \overline\be_2 p, \be_2 + \overline\be_1 p, p \rf$, where $\be_1 = \dfrac{5}{2}, \be_2 = 0, p=\dfrac{1}{2} $. Also $\beta_1 , \beta_2$ are unique because if we write $(y_1,y_2,p)=\lf \dfrac{5}{2}, \dfrac{5}{4}, \dfrac{1}{2} \rf$, then
\[
\beta_1=\frac{y_1-\bar{y}_2p}{1-|p|^2} \; \text{ and } \; \beta_2=\frac{y_2-\bar{y}_1p}{1-|p|^2}.
\]
Since $|\beta_1|+|\beta_2|<3$, $\lf \dfrac{5}{2}, \dfrac{5}{4}, \dfrac{1}{2} \rf \in \widetilde{\mathbb G_3}$. Clearly $\lf \be_1, \be_2 \rf \notin \mathbb G_2$, as $|\be_1| > 2$. By Theorem \ref{gn}, $\lf \dfrac{5}{2}, \dfrac{5}{4}, \dfrac{1}{2} \rf \notin \mathbb G_3$. So, $\mathbb G_3 \subsetneqq \widetilde{\mathbb G_3}$. In case of $n=3$, we choose $(\beta_1 ,\beta_2)=\lf \dfrac{5}{2},0 \rf$. Similarly for any $n>3$, we may choose $\lf \beta_1,\dots ,\beta_{n-1} \rf=\lf \dfrac{2n-1}{2},0,\dots, 0 \rf$ and $p=\dfrac{1}{2}$ to obtain a point $(y_1,\dots , y_{n-1},p)$, where $y_i=\beta_i+ \bar{\beta}_{n-i}p,$ for $i=1,\dots, n-1$, which is in $\widetilde{\mathbb G_n}$ but not in $\mathbb G_n$.
\end{proof}

In \cite{costara1}, Costara introduced the following rational function to characterize a point in $\gn$. For a point $s=(s_1,\dots,s_{n-1},p) \in \C^n$, let $f_s$ be the function defined by
\begin{equation}\label{costara function}
    f_s(z):= \dfrac{n(-1)^n p z^{n-1} + (n-1)(-1)^{n-1} s_{n-1} z^{n-2} +\cdots + (-s_1)}{n - (n-1)s_1 z + \cdots + (-1)^{n-1}s_{n-1} z^{n-1}}.
\end{equation}
It was shown in \cite{costara1} that the poles of $f_s$ lie outside $\overline\D$. This rational function characterizes the points in $\gn$ and $\gamn$ in the following way.

\begin{thm}[\cite{costara1}, Theorem 3.1 and 3.2]\label{costara}
    Let $s=(s_1,\dots,s_{n-1},p) \in \C^n$, and let $f_s$ be given by $\eqref{costara function}$. Then the following are equivalent:
    \begin{itemize}
        \item[(i)] $s=(s_1,\dots,s_{n-1},p) \in \gn \q ($or  $\in \gamn)$;
        \item[(ii)] $\q {\displaystyle \sup_{|z|\leq 1}|f_s(z)| < 1 } \q ($or $\leq 1)$.
    \end{itemize}
\end{thm}

\noindent For $r>0$, let $\D_r = \ls z \in \C : |z|<r \rs$. We denote by $\mathbb{G}_{n_r}$ and $\Gamma_{n_r}$ respectively the images of the set $\underbrace{\D_r \times \cdots \times \D_r}_{n- times}$ and its closure under the symmetrization map $\pi_n$, that is,
\[
\mathbb{G}_{n_r} = \pi_n ( \underbrace{\D_r \times \cdots \times \D_r }_{n- times} )
=\left\{ \left(\sum_{1\leq i\leq n} z_i,\sum_{1\leq i<j\leq n}z_iz_j,\dots, \prod_{i=1}^n z_i \right): \,z_i \in \D_r, i=1,\dots,n \right \} ,
\]
and
\[
\Gamma_{n_r} = \pi_n ( \underbrace{\overline\D_r \times \cdots \times \overline\D_r }_{n- times} )
=\left\{ \left(\sum_{1\leq i\leq n} z_i,\sum_{1\leq i<j\leq n}z_iz_j,\dots, \prod_{i=1}^n z_i \right): \,z_i \in \overline\D_r, i=1,\dots,n \right \} .
\]
Needless to mention that $\mathbb G_{n_r} \subseteq \gn$ and $\Gamma_{n_r}\subseteq \gamn$ for $0<r\leq 1$.

\begin{lemma}\label{schwarz gn prep}
    Let $(s_1,\dots,s_{n-1},p) \in \C^n$, and let $\lambda \in \overline{\D} \setminus \{ 0 \}$. Then the following are equivalent
    \begin{enumerate}
        \item $(s_1,\dots,s_{n-1},p) \in \mathbb{G}_{n_{|\lm|}}$ $(or\; \in \Gamma_{n_{|\lm|}})$.
        \item $\lf \dfrac{s_1}{\lm},\dots,\dfrac{s_{n-1}}{\lm^{n-1}}, \dfrac{p}{\lm^n} \rf \in \gn$ $(or\; \in \gamn)$
    \end{enumerate}
\end{lemma}

\begin{proof}

We prove for $\mathbb G_{n_{|\lm|}}$ and $\gn$. A proof for $\Gamma_{n_{|\lm|}}$ and $\gamn$ is similar.

    \noindent $(1) \imp (2)$. Suppose $(s_1,\dots,s_{n-1},p) \in \mathbb{G}_{n_{|\lm|}}$. Then there exist $z_1, z_2,\dots, z_n \in \D_{|\lm|}$ such that
    \[
    s_i(z)= \sum_{1\leq k_1 < k_2 \cdots < k_i \leq n} z_{k_1}\cdots z_{k_i} \quad \text{ and } p(z)=\prod_{i=1}^{n}z_i\,.
    \]
    Since $z_j \in \D_{|\lm|}$, we have $\dfrac{z_j}{\lm} \in \D$ for all $j=1,\dots, n$. Also
    \[\pi_n \lf \dfrac{z_1}{\lm}, \dots, \dfrac{z_n}{\lm} \rf =  \lf \dfrac{s_1}{\lm},\dots,\dfrac{s_{n-1}}{\lm^{n-1}}, \dfrac{p}{\lm^n} \rf .\]
    Thus $\lf \dfrac{s_1}{\lm},\dots,\dfrac{s_{n-1}}{\lm^{n-1}}, \dfrac{p}{\lm^n} \rf \in \gn$.\\

    \noindent $(2) \imp (1)$. Suppose $\lf \dfrac{s_1}{\lm},\dots,\dfrac{s_{n-1}}{\lm^{n-1}}, \dfrac{p}{\lm^n} \rf \in \gn$. Then there exist $w_1, w_2,\dots, w_n \in \D$ such that
    \[
    \dfrac{s_1}{\lm} = w_1+ \cdots + w_n, \q \dfrac{s_2}{\lm^2}= \sum_{1\leq j<k \leq n} w_j w_k,\: \dots,\;  \dfrac{p}{\lm^n} = \prod_{j=1}^{n} w_j.
    \]
    Clearly $\lm w_j \in \D_{|\lm|}$ for all $j = 1, \dots, n$, and
    $
    (s_1,\dots,s_{n-1},p) = \pi_n (\lm w_1, \dots, \lm w_n)
    $.
    Therefore, $(s_1,\dots,s_{n-1},p) \in \mathbb{G}_{n_{|\lm|}}$.
\end{proof}

\noindent  We now arrive at the desired main result of this article.

\begin{thm}
    Let $ \underline s^0= (s_1^0,\dots,s_{n-1}^0,p^0) \in \gn$ and let $\lm_0 \in \D \; \backslash \; \{0\}$. Suppose there exists an analytic function $\vp  :  \D \rightarrow  \gn $ such that  $\; \vp(0) = (0,\dots,0) $ and $\; \vp(\lm_0) = \underline s^0 $. Then
    \item[(1)]
    \[
    \sup_{z \in \overline \D} \left| \dfrac{n(-1)^n p^0 z^{n-1} + (n-1)(-1)^{n-1} s_{n-1}^0 z^{n-2} +\cdots + (-s_1^0)}{n - (n-1)s_1^0 z + \cdots + (-1)^{n-1}s_{n-1}^0 z^{n-1}} \right| \leq |\lm_0|.
    \]
    \item[(2)] The conditions $(2) - (11)$ of Theorem \ref{Schwarz Gn} hold for $y^0=\underline s^0 \in \gn$. Also, for $n=2$ each of these conditions is equivalent to the existence of such an interpolant $\psi$.
\end{thm}

\begin{proof} \item[(1)]
    Let $\vp  :  \D \rightarrow  \gn $ such that  $\; \vp(0) = (0,\dots,0) $ and $\; \vp(\lm_0) = \underline s^0 $. For $\omega \in \T$, consider the function $f_{\om} : \gn \longrightarrow \C$, defined as
    \[
    f_{\om} : (s_1,\dots,s_{n-1},p) \longmapsto \dfrac{n(-1)^n p {\om}^{n-1} + (n-1)(-1)^{n-1} s_{n-1} {\om}^{n-2} +\cdots + (-s_1)}{n - (n-1)s_1 \om + \cdots + (-1)^{n-1}s_{n-1} {\om}^{n-1}}
    \]
    It is evident from Lemma \ref{schwarz gn prep} that the function $f_{\om}$ is well defined and that $f_{\om}(\gn)= \D$. Then the function $g_{\om}= f_{\om} \circ \vp : \D \rightarrow \D$ is analytic, $g_{\om}(0)= (0,\dots,0)$ and
    \[
    g_{\om}(\lm_0)=  \dfrac{n(-1)^n p^0 {\om}^{n-1} + (n-1)(-1)^{n-1} s_{n-1}^0 {\om}^{n-2} +\cdots + (-s_1^0)}{n - (n-1)s_1^0 \om + \cdots + (-1)^{n-1}s_{n-1}^0 {\om}^{n-1}}.
    \]
    Therefore, by the classical Schwarz lemma for $\D$ we have
    \[
    \left|  \dfrac{n(-1)^n p^0 {\om}^{n-1} + (n-1)(-1)^{n-1} s_{n-1}^0 {\om}^{n-2} +\cdots + (-s_1^0)}{n - (n-1)s_1^0 \om + \cdots + (-1)^{n-1}s_{n-1}^0 {\om}^{n-1}} \right| \leq |\lm_0|.
    \]
    The above inequality is true for each $\om \in \T$, hence
    \[
    \sup_{\om \in \T} \left|  \dfrac{n(-1)^n p^0 {\om}^{n-1} + (n-1)(-1)^{n-1} s_{n-1}^0 {\om}^{n-2} +\cdots + (-s_1^0)}{n - (n-1)s_1^0 \om + \cdots + (-1)^{n-1}s_{n-1}^0 {\om}^{n-1}} \right| \leq |\lm_0|.
    \]
    Then, by Maximum principle
    \[
    \sup_{z \in \overline{\D}} \left|  \dfrac{n(-1)^n p^0 z^{n-1} + (n-1)(-1)^{n-1} s_{n-1}^0 z^{n-2} +\cdots + (-s_1^0)}{n - (n-1)s_1^0 z + \cdots + (-1)^{n-1}s_{n-1}^0 z^{n-1}} \right| \leq |\lm_0|.
    \]
    \item[(2)] Since $\vp$ maps $\D$ into $\Gn$, it is evident from Theorem \ref{Schwarz Gn} that conditions $(2) - (11)$ of Theorem \ref{Schwarz Gn}  follow necessarily if we replace $y^0$ by $\underline s^0$. Also, each of these conditions is equivalent to the existence of such an interpolant $\psi$ because, the conditions $(2) - (11)$ of Theorem \ref{Schwarz Gn} are all necessary and sufficient for the existence of an interpolant as in Theorem \ref{Schwarz Gn}. The proof is now complete.\\

\end{proof}

\section{Proofs of Lemma \ref{lemma-3.5} and Lemma \ref{lemma-3.6}}

\vspace{0.4cm}

    \noindent \textit{Proof of Lemma \ref{lemma-3.5}}.  We have that,
    	\[
    	 I-Z_j^*Z_j =
    	\begin{bmatrix}
    	1-\df{| y_j^0|^2}{\nj^2 |\lm_0|^2}-|w_j|^2 & &-\df{\bar{ y}_j^0}{\nj \bar{\lm_0}}w_j - \df{y_{n-j}^0}{\nj}\bar{w_j}\\
    	\\
    	-\df{ y_j^0}{\nj \lm_0}\bar{w_j} - \df{\bar{ y}_{n-j}^0}{\nj}w_j & & 1-\df{|y_{n-j}^0|^2}{\nj^2}-|w_j|^2
    	\end{bmatrix},
    	\]
where $w_j^2 = \dfrac{ y_j^0 y_{n-j}^0}{\nj^2 \lm_0} -  q^0 $. Therefore,
       \begin{equation}\label{app-1}
            \det(I-Z_j^*Z_j) =  \df{1}{\nj^2}\left( \nj^2 - \dfrac{| y_j^0|^2}{|\lm_0|^2} - |y_{n-j}^0|^2- 2\dfrac{| y_j^0  y_{n-j}^0 - \nj^2 q^0|}{|\lm_0|} + \nj^2 | q^0 |^2 \right).
            \end{equation}
        Since $\n \Phi_j(.,y) \n \leq |\lm_0|$, by $\eqref{Dj to 5}$, we have
        \begin{equation}\label{for 22 position}
            {n \choose j}\left|\dfrac{y_j^0}{\lm_0} - \bar y_{n-j}^0 \dfrac{q^0}{\lm_0}\right| + \left|\dfrac{y_j^0}{\lm_0} y_{n-j}^0 - {n \choose j}^2 \dfrac{q^0}{\lm_0} \right| \leq {n \choose j}^2 -|y_{n-j}^0|^2 .
        \end{equation}
        Then, by the equivalence of condition $(3)$ and $(5)$ of Theorem \ref{char G 3}, we have
        \begin{equation}\label{for det}
            \nj^2 - \dfrac{| y_j^0|^2}{|\lm_0|^2} - |y_{n-j}^0|^2- 2\dfrac{| y_j^0  y_{n-j}^0 - \nj^2 q^0|}{|\lm_0|} + \nj^2  \dfrac{|q^0|^2}{|\lm_0|^2} \geq 0,
        \end{equation}
        and by the equivalence of condition $(3)$ and $(3')$ of Theorem \ref{char G 3}, we have
        \begin{equation}\label{for 11 position}
            {n \choose j}\left|y_{n-j}^0 - \dfrac{\bar y_j^0}{\bar{\lm_0}} \dfrac{q^0}{\lm_0}\right| + \left|\dfrac{y_j^0}{\lm_0} y_{n-j}^0 - {n \choose j}^2 \dfrac{q^0}{\lm_0} \right| \leq {n \choose j}^2 -\dfrac{|y_j^0|^2}{|\lm_0|^2} .
        \end{equation}
        Therefore, by $\eqref{for 22 position}$, $\eqref{for det}$ and $\eqref{for 11 position}$ we have
        
        \begin{align*}
            & \det(I-Z_j^*Z_j) \geq 0, \\
            &(I-Z_j^*Z_j)_{11} = 1-\df{| y_j^0|^2}{\nj^2 |\lm_0|^2} -\df{| y_j^0 y_{n-j}^0 - \nj^2 q^0|}{\nj^2 |\lm_0|} \geq 0 ,\\
            \text{ and } \qq  &(I-Z_j^*Z_j)_{22} =1-\df{|y_{n-j}^0|^2}{\nj^2} -\df{| y_j^0 y_{n-j}^0 - \nj^2 q^0|}{\nj^2 |\lm_0|} \geq 0.
        \end{align*}
        Hence, $\det(I-Z_j^*Z_j)$ and the diagonal entries of $(I-Z_j^*Z_j)$ are all non-negative and so $\n Z_j \n \leq 1$.
        
Moreover, $\n \Phi_j(.,y) \n < |\lm_0|$ if and only if we have a strict inequality in $\eqref{for 22 position}$. Consequently, using the equivalence of conditions $(3)$, $(3')$ and $(5)$ of Theorem \ref{char G 3}, we have a strict inequality in both $\eqref{for 11 position}$ and $\eqref{for 22 position}$ if and only if $\n \Phi_j(.,y) \n < |\lm_0|$. Therefore, the diagonal entries and determinant of
        $(I-Z_j^*Z_j)$ are strictly positive and so $\n Z_j \n < 1$ if and only if $\n \Phi_j(.,y) \n < |\lm_0|$. Thus, $\n Z_j \n = 1$ if and only if $\n \Phi_j(.,y) \n = |\lm_0|$.
        
    \qed
    
    \vspace{0.3cm}
    
     \noindent \textit{Proof of Lemma \ref{lemma-3.6}}. We have that
	\begin{equation*}
	\mathcal{K}_{Z_j}(|\lm_0|) =\begin{bmatrix}
	[(1-|\lm_0|^2 Z_j^*Z_j)(1-Z_j^*Z_j)^{-1}]_{11} & [(1- |\lm_0|^2)(1 - Z_j Z_j^*)^{-1}Z_j]_{21} \\
	\\
	[(1 - |\lm_0|^2)Z_j^*(1 - Z_jZ_j^*)^{-1}]_{12} & [(Z_jZ_j^* - |\lm_0|^2)(1 - Z_jZ_j^*)^{-1}]_{22}
	\end{bmatrix}.
	\end{equation*}
	 For the given $Z_j$, we have
	 \begingroup
	 \allowdisplaybreaks
	 \begin{align*}
	 & (1-|\lm_0|^2 Z_j^*Z_j) =\begin{bmatrix}
	 1-\df{|y_j^0|^2}{\nj^2} -|\lm_0|^2|w_j|^2  & -\df{\lm_0\bar{y}_j^0 w_j}{\nj}- \df{|\lm_0|^2y_{n-j}^0\bar{w_j}}{\nj}   \\
	 \\
	 -\df{\bar\lm_0 y_j^0 \bar{w_j}}{\nj}- \df{|\lm_0|^2\bar{y}_{n-j}^0w}{\nj} & 1 - |\lm_0|^2|w_j|^2 -\df{|\lm_0|^2|y_{n-j}^0|^2}{\nj^2}
	 \end{bmatrix},\\
	 &(1-Z_j^*Z_j)^{-1} = \df{1}{\det(1-Z_j^*Z_j)}\begin{bmatrix}
	 1 - |w_j|^2 -\df{|y_{n-j}^0|^2}{\nj^2}   & \df{\bar{y}_j^0 w_j}{\nj \bar{\lm}_0}+ \df{y_{n-j}^0\bar{w_j}}{\nj}   \\
	 \\
	 \df{y_j^0 \bar{w_j}}{\nj\lm_0}+ \df{\bar{y}_{n-j}^0 w_j}{\nj} &  1-\df{|y_j^0|^2}{\nj|\lm_0|^2}- |w_j|^2
	 \end{bmatrix}.
	 \end{align*}
	 \endgroup
	 Then by a few steps of calculations we have
	 \begin{align}\label{app-4}
	 \nonumber &[(1-|\lm_0|^2 Z_j^*Z_j)(1-Z_j^*Z_j)^{-1}]_{11} \\
	 &=\df{1}{\det(1-Z_j^*Z_j)}\Bigg[ 1 - \df{|y_j^0|^2}{\nj^2} - \df{|y_{n-j}^0|^2}{\nj^2} + |q^0|^2 -
	 \df{|y_j^0y_{n-j}^0-\nj^2q^0|}{\nj^2}\Big(|\lm_0| +\frac{1}{|\lm_0|}\Big) \Bigg] .
	 \end{align}
	 Therefore,
	 \begin{align*}
	 &[\mathcal{K}_{Z_j}(|\lm_0|)\det (1-Z_j^*Z_j)]_{11} \\
	 & \qq =  1 - \df{|y_j^0|^2}{\nj^2} - \df{|y_{n-j}^0|^2}{\nj^2} + |q^0|^2 -
	 \df{|y_j^0y_{n-j}^0-\nj^2q^0|}{\nj^2}\Big(|\lm_0| +\frac{1}{|\lm_0|}\Big).
	 \end{align*}
	 Note that,
	 $$(1-Z_jZ_j^*)^{-1} = \df{1}{\det(1-Z_j^*Z_j)} \begin{bmatrix}
	 1 - |w_j|^2 -\df{|y_{n-j}^0|^2}{\nj^2}  &   \df{y_j^0 \bar{w}_j}{\nj \lm_0}+ \df{\bar{y}_{n-j}^0  w_j}{\nj}\\
	 \\
	 \df{\bar{y}_j^0 w_j}{\nj \bar{\lm}_0}+ \df{y_{n-j}^0 \bar{w}_j}{\nj} &  1-\df{|y_j^0|^2}{\nj^2|\lm_0|^2} - |w_j|^2
	 \end{bmatrix}.$$
	 Then, by simple calculation, we have
	 \begingroup
	 \allowdisplaybreaks
	 \begin{align*}
	 [(1-Z_jZ_j^*)^{-1}Z_j]_{21}
	 &= \df{1}{\det(1-Z_j^*Z_j)}\Bigg[\df{|y_j^0|^2w_j}{\nj^2 |\lm_0|^2} + \df{y_j^0y_{n-j}^0 \bar{w}_j}{\nj^2\lm_0} +
	 w_j - \df{|y_j^0|^2w_j}{\nj^2 |\lm_0|^2} -   w_j|w_j|^2 \Bigg]\\
	 &=\df{1}{\det(1-Z_j^*Z_j)}\Big[ w_j + \df{q^0}{\lm_0} \bar{w_j} \Big].
	 \end{align*}
	 \endgroup
	 Thus,
	 \[
	  [\mathcal{K}_{Z_j}(|\lm_0|)\det (1-Z_j^*Z_j)]_{12} = (1-|\lm_0|^2)\Big( w_j + \df{q^0}{\lm_0} \bar{w}_j \Big).
	  \]
     Similarly we have
     \begingroup
     \allowdisplaybreaks
     \begin{align*}
     &[Z_j^*(1-Z_jZ_j^*)^{-1}]_{12} \\
     &= \df{1}{\det(1-Z_j^*Z_j)}\Bigg[\df{|y_j^0|^2\bar{w}_j}{\nj^2 |\lm_0|^2} + \df{\bar{y}_j^0\bar{y}_{n-j}^0  w_j}{\nj^2\bar{\lm}_0} +
     \bar{w}_j - \df{|y_j^0|^2\bar{w}_j}{\nj^2 |\lm_0|^2} -   \bar{w}_j|w|^2 \Bigg]\\
     &=\df{1}{\det(1-Z_j^*Z_j)}\Big[ \bar{w}_j + \df{\bar{q}^0}{\bar{\lm}_0}  w_j
     \Big]\,,
     \end{align*}
     \endgroup
     and hence
     \[
     [\mathcal{K}_{Z_j}(|\lm_0|)\det (1-Z_j^*Z_j)]_{21} = (1-|\lm_0|^2)\Big( \bar{w}_j + \df{\bar{q}^0}{\bar{\lm}_0}  w_j \Big).\]
     Clearly    
      \begin{align*}
      (Z_jZ_j^* - |\lm_0|^2) 
      = \begin{bmatrix}
      \df{|y_j^0|^2}{\nj^2|\lm_0|^2} + |w_j|^2 -|\lm_0|^2  &   \df{y_j^0 \bar{w}_j}{\nj \lm_0} + \df{\bar{y}_{n-j}^0  w_j}{\nj}\\
      \\
      \df{\bar{y}_j^0 w_j}{\nj \bar{\lm}_0} +\df{y_{n-j}^0 \bar{w}_j}{\nj}  & |w_j|^2 +\df{|y_{n-j}^0|^2}{\nj^2} -|\lm_0|^2
      \end{bmatrix} .
      \end{align*}
      Simple calculations give us
     \begin{align}\label{app-5}
     \nonumber &[(Z_jZ_j^* - |\lm_0|^2)(1-Z_jZ_j^*)^{-1}]_{22}\\
     &= \df{1}{\det(1-Z_j^*Z_j)}\Bigg[\df{|y_j^0|^2}{\nj^2} +\df{|y_{n-j}^0|^2}{\nj^2} - |\lm_0|^2 +\df{|y_j^0y_{n-j}^0 - \nj^2q^0|}{\nj^2}\Big( |\lm_0|
     +\df{1}{ |\lm_0|}\Big) -\df{|q^0|^2}{|\lm_0|^2} \Bigg],
     \end{align}
     and consequently
     \begin{align*}
     &[\mathcal{K}_{Z_j}(|\lm_0|)\det (1-Z_j^*Z_j)]_{22}\\
     &= -|\lm_0|^2 - \df{|y_j^0|^2}{\nj^2} - \df{|y_{n-j}^0|^2}{\nj^2} +\df{|q^0|^2}{|\lm_0|^2} - \df{|y_j^0y_{n-j}^0-\nj^2q^0|}{\nj^2}\Big(|\lm_0|  +\frac{1}{ |\lm_0|}\Big).
     \end{align*}
     Therefore,
     \begingroup
     \allowdisplaybreaks
     \begin{align*}
     & \mathcal{K}_{Z_j}(|\lm_0|)\det (1-Z_j^*Z_j) = \\
     & \begin{bmatrix}
     1 - \dfrac{|y_j^0|^2}{\nj^2} - \dfrac{|y_{n-j}^0|^2}{\nj^2} + |q^0|^2  & \q &
     (1 - |\lm_0|^2)\Big( w + \dfrac{q^0}{\lm_0}  \bar{w} \Big) \\
     - \dfrac{|y_j^0y_{n-j}^0 - \nj^2q^0|}{\nj^2}
     \Big( |\lm_0| + \dfrac{1}{|\lm_0|} \Big) & \q & \q \\
     \\
     \q & \q &
     - |\lm_0|^2 + \dfrac{|y_j^0|^2}{\nj^2} + \dfrac{|y_{n-j}^0|^2}{\nj^2} - \dfrac{|q^0|^2}{|\lm_0|^2} \\
     (1 - |\lm_0|^2)\Big( \bar{w} + \dfrac{\bar{q}^0}{\bar{\lm}_0}  w \Big) & \q  & - \dfrac{|y_j^0y_{n-j}^0 - \nj^2q^0|}{\nj^2} \Big(  |\lm_0| + \dfrac{1}{ |\lm_0|} \Big)\\
     \end{bmatrix}.
     \end{align*}
     \endgroup
     Finally, by elementary calculation we have
     \begin{align*}
     \nonumber  -(k-k_j)(k-k_{n-j}) & = -\Bigg[2\frac{|y_j^0y_{n-j}^0 - \nj^2q^0|}{\nj^2} - |\lm_0| + \frac{|y_j^0|^2|\lm_0|}{\nj^2} + \frac{|y_{n-j}^0|^2}{\nj^2|\lm_0|} - \frac{|q^0|^2}{|\lm_0|} \Bigg] \times \\
     \nonumber & \qq  \Bigg[2\frac{|y_j^0y_{n-j}^0 - \nj^2q^0|}{\nj^2} - |\lm_0| + \frac{|y_j^0|^2}{\nj^2|\lm_0|} + \frac{|y_{n-j}^0|^2|\lm_0|}{\nj^2} - \frac{|q^0|^2}{|\lm_0|} \Bigg]\\
     & \qq =\det\left(\mathcal{K}_{Z_j}(|\lm_0|)\det(1-Z_j^*Z_j) \right),
     \end{align*}
     and the proof is complete.
     
     \qed

\vspace{1cm}

\section{Data Availability Statement}

\noindent (1) Data sharing is not applicable to this article as no datasets were generated or analysed during the current study.\\

\noindent (2) In case any datasets are generated during and/or analysed during the current study, they must be available from the corresponding author on reasonable request.

\end{document}